\documentclass[11pt,a4paper]{article} 

\usepackage{geometry}
\usepackage{amsmath,amssymb,amsthm,mathtools}
\usepackage{enumitem,bm}
\usepackage{hyperref}
\usepackage[nameinlink,capitalise]{cleveref}

\hypersetup{colorlinks=true,linkcolor=blue,citecolor=blue,urlcolor=blue}

\usepackage{tikz} 
\usetikzlibrary{shapes.geometric, arrows.meta, positioning, calc} 
\usetikzlibrary{fit, backgrounds}
\definecolor{figblue}{HTML}{1F4E79}
\definecolor{figamber}{HTML}{A85F00}

\usepackage{booktabs} 

\AtBeginDocument{%
  \setlength{\abovedisplayskip}{7pt plus 2pt minus 3pt}%
  \setlength{\belowdisplayskip}{7pt plus 2pt minus 3pt}%
  \setlength{\abovedisplayshortskip}{4pt plus 2pt minus 1pt}%
  \setlength{\belowdisplayshortskip}{5pt plus 2pt minus 2pt}%
}

\newtheorem{theorem}{Theorem}[section]
\newtheorem{lemma}[theorem]{Lemma}
\newtheorem{proposition}[theorem]{Proposition}

\newtheorem{problem}[theorem]{Problem}

\theoremstyle{definition}

\newtheorem{example}[theorem]{Example}

\newtheorem{remark}{Remark}
\theoremstyle{plain}

\newcommand{\eps}{\varepsilon}

\renewcommand{\le}{\leqslant}

\renewcommand{\ge}{\geqslant}

\begin{document}

\title{Optimal spectral supersaturation for cliques and odd cycles}

\author{Hongzhang Chen\thanks{School of Mathematics and Statistics, Gansu Center for Applied Mathematics, Lanzhou University, Lanzhou, Gansu, 730000, China. Email: \url{mnhzchern@gmail.com}.} 
\and Yongtao Li\thanks{Corresponding author. 
Yau Mathematical Sciences Center (YMSC), Tsinghua University, Beijing, 100084, China. Email: \url{ytli0921@hnu.edu.cn}.  Supported by the Postdoctoral Science Foundation of China (No.~2026M793384), the Postdoctoral Fellowship Program of China Postdoctoral Science Foundation (No.~GZC20261656) and the National Natural Science Foundation of China (No. 12601675).} 
}

\date{}

\maketitle

\vspace{-0.7cm}

\begin{abstract}
Motivated by the pioneering work of Bollob\'{a}s and Nikiforov, and of Ning and Zhai,  we study the spectral supersaturation problem, which asks how many copies of a fixed substructure $F$ are forced once a host graph exceeds the spectral Tur\'an threshold. Let $\lambda(G)$ denote the spectral radius of $G$, let $Y_{n,r,q}$ be the graph obtained from the Tur\'an graph $T_{n,r}$ by adding $q$ pairwise disjoint edges inside a largest part, and let $c(n,F)$ be the minimum number of copies of $F$ created by adding a single edge to $T_{n,r}$. Fang, Li, Lin and Ma proved that for every color-critical graph $F$ with $\chi (F)=r+1$, there exists a constant $\delta_F>0$ such that for all sufficiently large $n$ and all $1\le q\le \delta_F \sqrt{n}$,  the condition $\lambda(G)\ge\lambda(Y_{n,r,q})$ forces at least $q\, c(n,F)$ copies of $F$. The bound $q=O(\sqrt{n}\,)$ is tight up to a constant factor, in contrast to the linear order $n$ of the edge setting of Mubayi, Pikhurko and Yilma, but the exact constant $\delta_F$ remained unknown for any $F$.

In this paper, building on a structural result of Fang, Li, Lin and Ma, we determine the threshold $\delta_F$ when $F$ is a clique and an odd cycle. For every $r\ge2$, we denote $\delta_r :=(1-\tfrac1r)\sqrt2$ and prove that for every $\eps>0$, if $n$ is sufficiently large and $1\le q\le(\delta_r-\eps)\sqrt n$, then every $n$-vertex graph $G$ with $\lambda(G)\ge\lambda(Y_{n,r,q})$ contains at least $q\,c(n,K_{r+1})$ copies of $K_{r+1}$, and $\delta_r$ is best possible. For odd cycles, the threshold is  $1/\sqrt2$. For every $k\ge1$ and $\eps>0$, if $n$ is sufficiently large and $1\le q\le(1/\sqrt2-\eps)\sqrt n$, then every $n$-vertex graph $G$ with $\lambda(G)\ge\lambda(Y_{n,2,q})$ contains at least $q\,c(n,C_{2k+1})$ copies of $C_{2k+1}$, and $1/\sqrt2$ is best possible. Our results determine both the exact count of copies and the optimal range of $q$. The behavior in the spectral setting differs from the classical edge setting, in which  the range of $q$ is of order $n$ and the threshold is $1/r$ for cliques by Lov\'{a}sz and Simonovits, and $1/2$ for odd cycles by Pikhurko and Yilma.
\end{abstract}




\section{Introduction}

The Tur\'an-type problems ask for the maximum number of edges in a graph with no copy of a given graph $F$. 
 Let $T_{n,2}$ be the complete balanced bipartite graph.  
The well-known Mantel theorem says that if $G$ is an $n$-vertex graph with no triangle, then $e(G)\le e(T_{n,2})$, with equality if and only if $G=T_{n,2}$; see \cite{Bollobas78}.   
A central theme in extremal graph theory is the
supersaturation problem: once a graph has more edges than
the Tur\'an number of $F$, it must contain not just
one copy of $F$ but many, and one asks for the
precise count. The oldest instance is the
Erd\H{o}s--Rademacher problem on triangles, whose
study over several decades, from Erd\H{o}s
\cite{Erd1962a, Erdos1964}, Lov\'{a}sz--Simonovits \cite{LS1975, LS1983} to the complete solution of the
clique density problem by Razborov \cite{Raz2008}, 
Liu--Pikhurko--Staden \cite{LPS2020}, 
 Nikiforov \cite{Nik2011} and
Reiher \cite{Rei2016}, has produced some of the most influential
methods in the field.

Erd\H{o}s and Rademacher (see \cite{Erd1962a, Erdos1964}) extended Mantel's theorem by showing that  
if $e(G)> e(T_{n,2}) $, 
then $G$ contains at least $ \lfloor {n}/{2}\rfloor$ triangles.   
Moreover, 
Erd\H{o}s \cite{Erd1962a}
 showed that there exists a constant $\delta>0$ such that if  $n$ is sufficiently large and $1\le q<\delta n$, then every $n$-vertex graph with $\lfloor n^2/4\rfloor +q$
 edges  contains at least $q\lfloor {n}/{2}\rfloor$ triangles. 
Furthermore, 
Lov\'{a}sz and Simonovits \cite{LS1975} proved that for $1\le q < {n}/{2}$, if $G$ is an $n$-vertex graph with 
$\lfloor {n^2}/{4} \rfloor + q$ edges,   
then $G$ contains at least $q \lfloor {n}/{2}\rfloor$ 
triangles. 
The range $q <n/2$ is best possible.  Indeed, let $n$ be even and $q = {n}/{2}$. Then add $q+1$ extra edges 
to the larger partite set of $K_{\frac{n}{2}+1, \frac{n}{2}-1}$ in such a way that no triangle is formed 
among these $q+1$ edges.  
This yields a graph with ${n^2}/{4} + q$ edges and $(q+1)({n}/{2} -1) = 
{n^2}/{4} -1$ triangles, which is fewer than $q \lfloor {n}/{2}\rfloor$.  
We refer the reader to 
\cite{XK2021,LM2022-Erd-Rad,BC2023, MY2023} for recent developments. 

\smallskip 
A graph is called {\it color-critical} if it has an edge
whose removal lowers its chromatic number. This large
family, which includes cliques and odd cycles, plays a
central role in extremal graph theory. Let $T_{n,r}$
denote the $n$-vertex complete balanced $r$-partite
graph. Simonovits \cite{Sim1966} proved that for every
color-critical graph $F$ with chromatic number
$\chi(F)=r+1$, if $n$ is large enough and $G$ is an
$n$-vertex graph with no copy of $F$, then
$e(G)\le e(T_{n,r})$, with equality if and only if
$G=T_{n,r}$. 
Extending this result, Mubayi \cite{Mubayi2010} proved a
supersaturation theorem for all color-critical graphs. 
Let $c(n,F)$ denote the
smallest number of copies of $F$ created by adding an 
 edge to $T_{n,r}$.  
 Explicitly, $c(n,K_{r+1})=(1+o(1))(n/r)^{\,r-1}$ and $c(n,C_{2k+1})=(1+o(1))(n/2)^{\,2k-1}$.
In particular, both reduce to $\lfloor n/2\rfloor$ when $r=2$ and $k=1$, recovering the
Erd\H{o}s--Rademacher count for triangles. 

\begin{theorem}[Mubayi \cite{Mubayi2010}] \label{thm-Mubayi} 
Let $F$ be a color-critical graph with $\chi(F)=r+1$.
There exists a constant $\delta_F>0$ such that for all sufficiently large $n$, and for every integer $q$ with  
$1\le q\le\delta_F \,n$, if $G$ is an $n$-vertex graph 
with $e(G)\ge e(T_{n,r})+q$, 
then  $G$ contains at least $q\, c(n,F)$ copies of $F$.
\end{theorem}

Let $\mathcal{T}_{n,r,q}$
be the family of graphs obtained from the Tur\'{a}n graph  $T_{n,r}$ by
adding $q$ edges. 
Pikhurko and Yilma \cite{PY2017} later
sharpened Theorem \ref{thm-Mubayi} by describing the extremal graphs, showing further that under the hypothesis of Theorem \ref{thm-Mubayi}, 
 if $G$ contains the fewest copies of $F$, then $G\in\mathcal{T}_{n,r,q}$. 
Throughout this paper, 
we write $\delta_{\rm edge}(F)$ for the {\it largest constant} $\delta_F$ for
which Theorem \ref{thm-Mubayi} holds; this determines the optimal range  
of $q$ in the edge setting. 
The problem of determining $\delta_{\rm edge}(F)$ 
for cliques $F$ is a well-known question of Erd\H{o}s,
which was solved by  
Lov\'asz and Simonovits
\cite{LS1983}. 

\begin{theorem}[Lov\'asz--Simonovits
\cite{LS1983}]  \label{thm-edge-clique}
Let $r\ge 2$ be an integer. Then 
$$\delta_{\rm edge}(K_{r+1})= \frac{1}{ r}.$$ 
\end{theorem}

Apart from cliques, 
Pikhurko and Yilma \cite{PY2017} 
provided a powerful approach that allows one to determine the exact threshold $\delta_{\rm edge}(F)$ for a large number of color-critical
graphs $F$, including odd cycles, cliques with one edge removed, and complete bipartite
graphs plus an edge.

\begin{theorem}[Pikhurko--Yilma \cite{PY2017}] 
\label{thm-edge-cycle}
For every $k\ge 1$, we have 
$\delta_{\rm edge}(C_{2k+1})= 1/ 2$; 
let $r\ge 2$ and $K_{r+2}-e$ be obtained from $K_{r+2}$ by removing an edge. Then 
$\delta_{\rm edge}(K_{r+2}-e)={(r-1)}/{r^2}$. 
\end{theorem}

\subsection{Spectral supersaturation problems}

Let $\lambda(G)$ be the spectral radius of the adjacency matrix of $G$, and let $\texttt{\#}F (G)$ be the number of copies of $F$ in $G$. 
Spectral Tur\'{a}n theory 
replaces the edge count $e(G)$ by the
spectral radius $\lambda(G)$ and looks for analogues of classical
Tur\'an-type results; see the survey  \cite{Nik2011survey}. 
In 1986, Wilf \cite{Wil1986} showed that every $K_{r+1}$-free graph $G$ on $n$ vertices satisfies $\lambda (G) \le (1- \frac{1}{r})n$; see, e.g., \cite{Niki2009cpc, Niki2009jgt, KN2014} for related results.  
The Simonovits theorem for color-critical graphs was generalized by Nikiforov \cite{Nik2009EJC}: if $F$ is a color-critical graph with $\chi(F)=r+1$, and
$G$ is an $n$-vertex $F$-free graph with $n$ large, then
$\lambda(G)\le\lambda(T_{n,r})$, 
with equality if and only if $G=T_{n,r}$. This spectral version implies the edge version in a non-trivial way; see  \cite{LiuNing2026}. 
Spectral supersaturation asks how many copies of $F$ are forced in $G$ once
$\lambda(G)$ passes the spectral threshold $\lambda(T_{n,r})$.  
Counting substructures under spectral conditions
has attracted much attention in recent years. 
 In 2007, 
Bollob\'{a}s 
and Nikiforov \cite{BN2007jctb} proved that   
\begin{equation} \label{eq-BN}
\texttt{\#}K_{r+1}(G) \ge \Big( \frac{\lambda (G)}{n} - 1 + \frac{1}{r} \Big) 
\frac{r(r-1)}{r+1} \Big( \frac{n}{r}\Big)^{r+1}. 
\end{equation}
This bound turns the spectral radius into a lower bound on the number of
$(r+1)$-cliques, and it is exactly the inequality that
our forthcoming Theorem \ref{thm-continuous} sharpens just above the Tur\'an threshold. 
In 2023, Ning and Zhai \cite{NZ2021} proved that if $G$ is a graph on $n$ vertices with  $\lambda (G) \ge \lambda (T_{n,2})$, 
then $G$ contains at least 
$\left\lfloor \frac{n}{2}\right\rfloor -1$ triangles, 
unless $G=T_{n,2}$.   
In recent years, a number of results on spectral
supersaturation have been obtained, 
including triangles \cite{LFP2024, ZZ2026, CLT2026},  
$4$-cycles \cite{NZ2021b,LLZ2024-book-4-cycle},
 books \cite{ZL2022jgt, ZLL2026, ZYZZ2026, WLS2026},  
 color-critical graphs \cite{FLLM2025, FLZ2026, FLZ2026-tripartite, CL2026} 
and some classical bipartite graphs \cite{LLZ2026-Spec-Sido}.

\smallskip 
To find a spectral analogue of Theorem~\ref{thm-Mubayi},
one needs a candidate graph such that  $\lambda (G)$ above the candidate could force at least $q\,c(n,F)$ copies of $F$. The
natural choice is $Y_{n,r,q}$, the graph obtained from
$T_{n,r}$ by adding $q$ \textit{pairwise
disjoint edges} inside one largest part. 
Among all graphs 
built from $T_{n,r}$ by adding $q$ edges, $Y_{n,r,q}$ attains 
the smallest spectral radius; this was proved for all $r\ge2$ and
$q\le \frac{n}{100r}$ in \cite{FLLM2025}. 
Thus $\lambda(G)\ge\lambda(Y_{n,r,q})$ is the {\it weakest}\, spectral
condition under which one can hope to force $q\,c(n,F)$
copies. For triangles, Li, Feng and Peng
\cite{LFP2024} established the first result of this kind:
if $1\le q\le\frac{1}{9}\sqrt n$ and $G$ is an $n$-vertex graph
with $\lambda(G)\ge\lambda(Y_{n,2,q})$, then
$\texttt{\#}K_3(G)\ge q\lfloor n/2\rfloor$. 
For color-critical graphs, Fang, Li, Lin and Ma \cite{FLLM2025} proved a general result.

\begin{theorem}[Fang--Li--Lin--Ma \cite{FLLM2025}] \label{thm-Y}
    Let $F$ be a color-critical graph
with $\chi (F)=r+1$. There exists $\delta_F >0$ such that 
if $n$ is sufficiently large, $1\le q \le  \delta_F \sqrt{n}$, 
and $G$ is an $n$-vertex graph with
\begin{equation*} 
     \lambda (G) \ \ge  \ 
      \lambda (Y_{n,r,q}), 
\end{equation*}
then $\texttt{\#} F(G) \ge q\, c(n,F)$.
Under the above conditions, if $G$ minimizes $\texttt{\#} F(G)$, then $G\in \mathcal{T}_{n,r,q}$.
\end{theorem}

The bound $q=O(\sqrt{n}\,)$  in Theorem \ref{thm-Y} is tight up to a constant factor, which reveals a
phenomenon different from the edge setting of Theorem~\ref{thm-Mubayi}. 
A spectral condition of this type is not comparable with the edge condition of Theorem \ref{thm-Mubayi}: a graph can satisfy $\lambda(G)\ge\lambda(Y_{n,r,q})$ while having fewer than $e(Y_{n,r,q})$ edges. 
Since a small number of high-degree vertices forces a large spectral radius, the spectral supersaturation statement does not follow from its edge counterpart, and this is closely related to why the range of $q$ shrinks from order $n$ to order $\sqrt n$.  

\smallskip 
Let $\delta_{\rm spec} (F)$ be the {\it largest constant} $\delta_F$  such that Theorem \ref{thm-Y} holds. 
Despite the aforementioned results, no exact value of $\delta_{\rm spec} (F)$ has been determined in the spectral
setting for any color-critical graph $F$.   
Unlike the edge setting, which comes from a simple count of edges, the spectral condition $\lambda(G)\ge\lambda(Y_{n,r,q})$ works in a different way, as adding $q$ disjoint edges increases the spectral radius by nearly $2q/n$. So the crucial comparison occurs at the second-order term of the spectral expansion of $\lambda(Y_{n,r,q})$. This also explains why the optimal range of $q$ in the spectral setting drops from about $n$ to about $\sqrt n$, and why the optimal constant is not a simple fraction. Consequently, determining the exact value of $\delta_{\rm spec}(F)$ appears more subtle and difficult, and it requires precise control of the Perron vector, instead of the rough estimates that are enough for the order
of magnitude.

\subsection{Main results}

The main contribution of this paper is to establish 
spectral supersaturation results by determining the exact thresholds $\delta_{\rm spec}(K_{r+1}) =(1-\frac{1}{r})\sqrt{2} $ and $\delta_{\rm spec}(C_{2k+1}) = \frac{1}{\sqrt{2}}$; see Table \ref{tab:edge-vs-spec}.     
Our results can be viewed as the spectral counterparts of Theorems \ref{thm-edge-clique} and \ref{thm-edge-cycle}.  
We present these results in the following precise statement. 
 For a graph $F$ with $\chi (F)=r+1$, 
 we recall that $c(n,F)$ denotes the minimum number of copies of $F$ created by adding a single edge to the Tur\'{a}n graph $T_{n,r}$.

\begin{theorem} 
\label{thm-cliques}
For every $r\ge 2$, write  
$ \delta_r :=(1-\frac{1}{r})\sqrt{2}$.  
 For every  $\varepsilon>0$, if $n$ is sufficiently large and $ 1\le q\le (\delta_r -\varepsilon)\sqrt n$, 
then every $n$-vertex graph $G$ with 
$ \lambda(G)\ge \lambda(Y_{n,r,q})$ 
satisfies 
$$ \texttt{\#}K_{r+1}(G)\ \ge \ q\,c(n,K_{r+1}),$$
 with equality  only if $G\in \mathcal{T}_{n,r,q}$. 
Moreover, the constant $\delta_r$ is best possible. 
\end{theorem}

We next turn to another typical family of color-critical graphs, the odd cycle $C_{2k+1}$ for all $k\ge 1$.
For this family of graphs, the threshold is an absolute constant  $1/\sqrt2$, which is independent of the length
$2k+1$, in contrast to the clique threshold $\delta_r$ above, which depends on $r$.

\begin{theorem}\label{thm-odd-cycles}
For every $k\ge 1$ and $\eps>0$, write $\delta_{\rm opt} := 1/\sqrt{2}$. If $n$ is sufficiently large and $1\le q\le (\delta_{\rm opt}-\eps)\sqrt n$, then every $n$-vertex graph  $G$ with $\lambda(G)\ge \lambda(Y_{n,2,q})$ satisfies 
$$ \texttt{\#}\, {C_{2k+1}}(G)\ \ge \ q\,c(n,C_{2k+1}),$$
 with equality only if $G\in \mathcal{T}_{n,2,q}$. Moreover, the constant $\delta_{\rm opt}$ is best possible. 
\end{theorem}

Up to now, we summarize the spectral supersaturation results in Table \ref{tab:edge-vs-spec}. 

\begin{table}[htbp]
\centering 
\renewcommand{\arraystretch}{1.3}
\begin{tabular}{@{}ccc@{}}
\toprule
  & Edge setting
  & Spectral setting \\
\midrule
Assumption 
  & $e(G)\ \ge \  e(T_{n,r})+q$
  & $\lambda(G) \  \ge \  \lambda(Y_{n,r,q})$ \\
Range of $q$
  & $q\le \delta_F\, n$
    (Mubayi \cite{Mubayi2010})
  & $q\le \delta_F \sqrt{n}$
    (Fang--Li--Lin--Ma \cite{FLLM2025}) \\
Conclusion
  & $\texttt{\#}F(G) \ \ge \  q\,c(n,F)$
  & $\texttt{\#}F(G) \ \ge  \ q\,c(n,F)$ \\
\hline
Threshold for $K_{r+1}$
  & $\tfrac{1}{r}$
    (Lov\'asz--Simonovits \cite{LS1975})
  & $\bigl(1-\tfrac{1}{r}\bigr)\sqrt{2}$
    (Theorem~\ref{thm-cliques}) \\[2pt]
    \hline 
Threshold for $C_{2k+1}$
  & $\tfrac{1}{2}$
    (Pikhurko--Yilma \cite{PY2017})
  & $\tfrac{1}{\sqrt{2}}$
    (Theorem~\ref{thm-odd-cycles}) \\
\bottomrule
\end{tabular}
\caption{Supersaturation in the edge setting versus the spectral setting.}
\label{tab:edge-vs-spec}
\end{table}

In contrast with the edge setting, 
Theorems \ref{thm-cliques} and \ref{thm-odd-cycles} differ from the edge setting in two ways: the scale of 
$q$ and the value of the threshold. 
First, the conclusion of Theorem \ref{thm-Mubayi} holds for all $q\le \delta_F n$, of order $n$, while the range in the spectral setting is of order $\sqrt n$. Second, in Theorems \ref{thm-edge-clique} and \ref{thm-edge-cycle}, we see that 
$\delta_{\mathrm{edge}}(K_{r+1})=1/r$ and
$\delta_{\mathrm{edge}}(C_{2k+1})=1/2$, whereas our
spectral thresholds are $\delta_{\mathrm{spec}}(K_{r+1}) = (1- \frac{1}{r})\sqrt2$ and $\delta_{\mathrm{spec}}(C_{2k+1}) = 1/\sqrt2$.
So the spectral threshold must be determined by a  different approach. 
The two settings behave in opposite directions as $r$ grows: the edge threshold
$1/r$ decreases to $0$, while the spectral threshold
$(1- \frac{1}{r})\sqrt2$ increases to $\sqrt2$.  
 In the edge setting, 
$\delta_{\mathrm{edge}}(F)$ depends on the fine structure of $F$, not only on
its chromatic number: it equals $1/r$ for $K_{r+1}$,  but $(r-1)/r^2$ for 
$K_{r+2}-e$. We speculate that the spectral threshold  $\delta_{\mathrm{spec}}(F)$ may depend only on the chromatic number of $F$, if true, this would be a new 
feature with no edge analogue. 

\smallskip 
The scale $\sqrt n$ has a transparent spectral
meaning. Adding $q$ edges to Tur\'{a}n's graph $T_{n,r}$ raises the spectral radius by ${2q}/{n}$ at first order,
regardless of how the added edges are distributed. 
{\it The distribution becomes visible only at second order term}, where a star gains an extra
${q^2 (1- \frac{1}{r})^{-1}}/{n^2}$, while an unbalanced
pair of partite sets loses ${2 (1- \frac{1}{r})}/{n}$. These two effects balance when $q$ has order $\sqrt n$, and
the constant $(1-\tfrac1r)\sqrt2$ records the
point of balance. Thus, any first-order estimate is too coarse to detect the threshold of $q$, which is determined 
at the second-order estimate.

\smallskip 
The supersaturation result in Theorem \ref{thm-Y} compares $G$ with the candidate graph $Y_{n,r,q}$ and
lets $q$ range up to $O(\sqrt{n}\,)$.  
A more refined goal of this paper is to bound
$\texttt{\#}F(G)$ directly by the increment 
$\lambda(G)-\lambda(T_{n,r})$. Adding $q$ disjoint edges raises the spectral radius 
by about $2q/n$, so the range
$q=O(\sqrt n\,)$ corresponds to an increment of order
${n}^{-1/2}$. As an application of our approach, we establish the following continuous version of Theorem \ref{thm-Y}.

\begin{theorem} 
\label{thm-continuous}
    For any color-critical graph $F$ with $\chi (F)=r+1$, there exists $\delta_F >0$ such that for all sufficiently large $n$, and all real $C$ with $0<C\le \delta_F {n}^{-1/2}$, if $G$ is an $n$-vertex graph with 
    $$ \lambda (G)\ge \lambda (T_{n,r}) + C, $$  then 
    \begin{equation*}  
    \texttt{\#}F(G) \ge 
    \Big(\frac{C}{2}- o(1) \Big) 
     n \cdot c(n,F).
    \end{equation*}
    Moreover, this bound is tight, as witnessed by the graph $Y_{n,r,q}$ for all small $q$. 
\end{theorem}

\begin{remark}
Taking $F=K_{r+1}$, Theorem~\ref{thm-continuous} improves the Bollob\'as--Nikiforov bound (\ref{eq-BN}) in a small range above $\lambda (T_{n,r})$. The latter states that if 
$\lambda(G) = \lambda(T_{n,r})+\beta n$ for some $\beta>0$,
then
\[
    \texttt{\#}K_{r+1}(G)
    \ge\beta\,\frac{r(r-1)}{r+1}\Bigl(\frac nr\Bigr)^{r+1}
    =\beta n\, \frac{r-1}{r(r+1)}\cdot n\Bigl(\frac nr\Bigr)^{r-1}.
\]
Both bounds are linear in the spectral increment 
$\lambda(G)-\lambda(T_{n,r})$. Take $\beta \le \delta n^{-3/2}$
with a constant $0<\delta\le\delta_F$, so that the spectral increment is at most $\delta n^{-1/2}$ and Theorem~\ref{thm-continuous}
applies. By definition, we have $c(n,K_{r+1}) = (1+o(1))(\frac{n}{r})^{r-1}$. 
Consequently, our bound in Theorem \ref{thm-continuous} implies   
\[
    \texttt{\#}K_{r+1}(G)
    \ge\Bigl(\frac{\delta}{2}-o(1)\Bigr)
        {\sqrt n}\,\Bigl(\frac nr\Bigr)^{r-1}. 
\]
In contrast, the Bollob\'as--Nikiforov bound gives 
\[  \texttt{\#}K_{r+1}(G)
    \ge\frac{(r-1)\delta}{r(r+1)}
   {\sqrt n}\,\Bigl(\frac nr\Bigr)^{r-1}.  \]
   The former exceeds the latter by a factor
$\frac{r(r+1)}{2(r-1)}$, which is 
linearly large in terms of $r$. 
In particular, in the range $0 < \lambda (G) - \frac{n}{2} \le \delta n^{-1/2}$ for some $\delta >0$, we get  
$\texttt{\#}K_3(G) \ge \big(\frac{1}{4} - o(1) \big) n^2 \big(\lambda (G) - \frac{n}{2} \big)$. 
\end{remark}

Counting cliques above the classical Tur\'an threshold has a long history in the edge setting, from Erd\H{o}s \cite{Erd1962a, Erdos1964},  Lov\'asz--Simonovits \cite{LS1975} and Razborov \cite{Raz2008} to 
Liu--Pikhurko--Staden \cite{LPS2020} for triangles, Nikiforov \cite{Nik2011} for $K_4$, and Reiher \cite{Rei2016}
for all cliques. Theorem \ref{thm-continuous} can be viewed as a spectral counterpart in a narrow range just above $\lambda(T_{n,r})$, where the number of cliques grows linearly in the spectral increment
$\lambda(G)-\lambda(T_{n,r})$. 
For a general graph $F$, one expects to count the number of copies of $F$ in any graph $G$ with $\lambda (G) \ge \lambda (T_{n,r}) + C$ for any universal constant $C>0$.

\subsection{Proof overview and organization}

Throughout the paper, we assume that $r$ is fixed and $n$ is sufficiently large,  and all constants hidden in 
$O_r(\cdot)$ may depend on $r$, but not on $n$.    
Our proofs of Theorems \ref{thm-cliques} and \ref{thm-odd-cycles} are based on the spectral supersaturation-stability method. First, we reduce any extremal graph to one that is close to a nearly balanced complete $r$-partite graph. 
After the stability reduction, the graph differs from a nearly balanced complete
$r$-partite graph by only $O(\sqrt n\,)$ edge changes. Next, a key ingredient is the following second-order estimate on the spectral radius.

\begin{theorem}[The second-order estimate]
\label{thm:local-spectral}
For any $r\ge 2$, write $\alpha:=\frac{r-1}{r}$, let $K=K_r(n_1,\dots,n_r)$ be an
$n$-vertex complete $r$-partite graph, and let $G$ be obtained from $K$ by adding
$s$ class-edges and deleting $h$ cross-edges, where $s,h=O_r(\sqrt n)$. Assume
$D:=e(T_{n,r})-e(K)=O_r(\sqrt n)$. Then
\begin{equation*}
        \lambda(G)\le \lambda(T_{n,r})-\frac{2\alpha D}{n}+\frac{2(s-h)}{n}
        +\frac{(s^2+h^2)\alpha^{-1}}{n^2}+O_r(n^{-5/4}). 
\end{equation*}
\end{theorem}

This second-order estimate controls the spectral radius of a complete $r$-partite graph after $O(\sqrt n)$ class-edge additions 
and cross-edge deletions, 
which appears in many
spectral stability arguments and may be useful for other spectral extremal problems where first-order estimates are too coarse.

\paragraph{Proof overview.}
For Theorem \ref{thm-cliques}, we apply a structural result of Fang, Li, Lin and Ma (Theorem \ref{thm:FLLM}), which transforms any extremal graph satisfying $\lambda (G)\ge \lambda (Y_{n,r,q})$ and 
$\texttt{\#}K_{r+1}(G) \le q\, c( n,K_{r+1})$ into a graph that is built from a nearly balanced complete $r$-partite graph $K$ by adding $s$ class-edges and deleting $h$ cross-edges, where $s, h= O_r(\sqrt n\,)$ and $D:=e(T_{n,r})-e(K)=O_r(\sqrt n\,)$ (Lemma \ref{lem:reduction}). For such graphs, we then prove two precise estimates. 
The first is a
second-order expansion of $\lambda (G)$ as stated in Theorem \ref{thm:local-spectral}, 
which we get by studying the Perron vector of $K$ up to second order. Comparing with $\lambda(Y_{n,r,q})=\lambda(T_{n,r})+2q/n+O_r(n^{-3/2})$ (Lemma \ref{lem:Y-expansion}) gives the key inequality:  
\begin{equation} \label{eq-overview}
        0\le \lambda(G)-\lambda(Y_{n,r,q})
        \le \frac{2(s-q-h)}{n}-\frac{2\alpha D}{n}
            +\frac{(s^2+h^2)\alpha^{-1}}{n^2}+o(n^{-1}).  
\end{equation}  
The second is a counting
lemma (Lemma \ref{lem:local-counting}): each added class-edge of $G$ creates about $c(n,K_{r+1})$ copies of
$K_{r+1}$, up to small losses coming from the imbalance $D$ and the deletions $h$, that is,  
\[         \texttt{\#} K_{r+1}(G)\ge s\,c(n,K_{r+1})-C_r\,s \left(\sqrt D+h+1 \right)\,n^{r-2}.  \]
Combining with the assumption $\texttt{\#}K_{r+1}(G) \le q\, c( n,K_{r+1})$, we get $\sqrt{D}+ h +1 \ge \kappa \frac{(s-q) n}{s}$ for a constant $\kappa=\kappa(r) >0$.  
Next, we show that the right-hand side of (\ref{eq-overview}) is negative for $s\neq q$. 
In the case $s<q$, 
the spectral estimate (\ref{eq-overview}) gives $\lambda (G)- \lambda (Y_{n,r,q}) \le \frac{2(s-q)}{n} + \frac{s^2\alpha^{-1}}{n^2}+ o(n^{-1}) <0$. 
In the case $s> q$, we write $s:=q+\ell$ for some $\ell \ge 1$.  Then $\sqrt{D} + h +1 \ge \kappa \frac{\ell n}{s}$, which yields $\frac{2h + 2\alpha D}{n} \ge \frac{2\alpha \kappa \ell}{s} - \frac{2\alpha }{n}$. 
So $\lambda (G)- \lambda (Y_{n,r,q}) \le \frac{2\ell}{n} - \frac{2\alpha \kappa \ell }{s} + O_r(n^{-1}) <0$ since  $s=O_r(\sqrt{n}\,)$.  
Now, we consider the case $s=q$. 
If $h\ge 1$ or $D\ge 1$, then one can verify that $\frac{-2h}{n} - \frac{2\alpha D}{n} + \frac{h^2 \alpha^{-1}}{n^2} \le 
 - \frac{2\alpha}{n} + o(n^{-1})$. 
 Thus $\lambda (G) - \lambda (Y_{n,r,q}) \le - \frac{2\alpha}{n} + \frac{q^2 \alpha^{-1}}{n^2} + o(n^{-1}) <0$, a contradiction. 
 Consequently, we must have $s=q$ and $h=D=0$. 
 Hence $G$ is obtained from the Tur\'{a}n graph $T_{n,r}$ by adding $q$ edges. 

To see that the constant $\delta_r$ cannot be improved, fix
$q=(\delta_r+\eps)\sqrt n$ and let $H=H_{n,r,q}$ be built
from the unbalanced complete $r$-partite graph
$K_r(m+1,m-1,m,\ldots,m)$ (where $n=rm$, with one part
enlarged and one shrunk by one vertex) by adding a star
with $q$ edges inside the part of size $m+1$. Shrinking a
part lowers the count, so $H$ has only
$q(m-1)m^{r-2}<q\,c(n,K_{r+1})$ copies of $K_{r+1}$. Yet
the added star raises the spectral radius enough that
$\lambda(H)>\lambda(Y_{n,r,q})$ once
$q>(\delta_r+\eps)\sqrt n$ (Example \ref{exam-clique}). Thus, 
above the threshold $\delta_r$, the spectral condition $\lambda(H)>\lambda(Y_{n,r,q})$ no longer
forces $q\,c(n,K_{r+1})$ copies. 
The threshold $\delta_r=\sqrt2\,\alpha$ is exactly the point at which the second-order term gain $q^2\alpha^{-1}/n^2$ 
from concentrating the added edges into a star balances the loss $2\alpha /n$ from making the underlying $r$-partite graph unbalanced, that is, $q^2 \alpha^{-1}=2\alpha n$, so $q=\sqrt{2} \alpha  \cdot \sqrt{n}$.

The proof of Theorem \ref{thm-odd-cycles} follows the same
plan. The stability input is Lemma \ref{lem-stability-odd},
the second-order spectral estimate is
Lemma \ref{lem:local-spectrum-odd}, and the counting step is
Lemma \ref{lem:local-counting-odd}, which shows that each
added class-edge lies in about $c(n,C_{2k+1})$ copies of
$C_{2k+1}$. The same three cases $s<q$, $s=q$, $s>q$ give
the bound, and the construction in
Example \ref{exam-odd-cycle} shows that $1/\sqrt2$ is best
possible.

\paragraph{Organization.}  
In Section \ref{sec:prelim}, we introduce the tools: 
a second-order expansion for the spectral radius of a complete $r$-partite graph, 
the expansion of $\lambda(Y_{n,r,q})$, 
and a sharp construction obtained by adding a star. 
In Section \ref{sec-1-8-and-1-7}, 
we first prove the key estimate in Theorem \ref{thm:local-spectral}, and then we apply it to prove Theorem \ref{thm-continuous}. 
In Section \ref{sec:clique}, we prove
Theorem \ref{thm-cliques}, and in Section \ref{sec:odd}, we prove
Theorem \ref{thm-odd-cycles}. 
In Section \ref{sec-concluding-remark}, we conclude with some open spectral extremal problems for interested readers.

\section{Preliminaries} 

\label{sec:prelim}

We need the following structural theorem of 
Fang, Li, Lin and Ma \cite{FLLM2025}.

\begin{theorem}[Fang--Li--Lin--Ma \cite{FLLM2025}]
\label{thm:FLLM}
Let $0<\varepsilon_1<1/2$, and let $F$ be a color-critical graph  with $\chi(F)=r+1$. There are constants
$\delta_F,\eta_F>0$ such that the following holds for every sufficiently large $n$.
If $  1\le q\le \delta_F n$ 
and $G$ is an $n$-vertex graph satisfying
\[
        \lambda(G)\ge \lambda(T_{n,r})
\]
and
\[
        \texttt{\#} F(G)\le (q+\varepsilon_1)c(n,F),
\]
then $G$ can be obtained from a complete $r$-partite graph $  K_r(n_1,n_2,\ldots,n_r)$ 
by adding $s$ class-edges and deleting $h$ cross-edges, where $  h\le 2s\le {4q}/{\eta_F} $ and 
$|n_i-n_j|\le 4\sqrt{s}$ for all $i\neq j$.  
\end{theorem}

The following lemma bounds 
 the sum of squares of the degrees in a graph of given size.  
  
 \begin{lemma}[See \cite{NZ2021b}] \label{lem-deg-squ}
If $G$ is a graph with $m$ edges, then  
\[  \sum_{v\in V(G)}d^2(v) \le m^2 +m. \]  
 \end{lemma}

\subsection{The spectral radius of complete multipartite graphs}

First, we estimate the spectral radius of a nearly balanced complete
$r$-partite graph.

\begin{theorem} 
\label{thm:multipartite-expansion}
Let $  K=K_r(n_1,n_2,\ldots,n_r)$ 
be a complete $r$-partite graph on $n$ vertices. 
Write $\alpha= \frac{r-1}{r}$ and $  n_i=\frac nr+a_i$, where $\sum_{i=1}^r a_i=0$, and assume $ B:=\sum_{i=1}^r a_i^2=O_r(\sqrt n)$. 
Then
\begin{equation}
        \lambda(K)
    =\alpha n-\frac{\alpha B}{n}+O_r(n^{-5/4}). 
        \label{eq:multipartite-expansion}
\end{equation}
Moreover, if $ D:=e(T_{n,r})-e(K)$, 
then
\begin{equation*}
        \lambda(K)
        =\lambda(T_{n,r})-\frac{2\alpha D}{n}+O_r(n^{-5/4}). 
\end{equation*}
\end{theorem}

\begin{proof} 
Denote $\lambda =\lambda (K)$. 
By the Perron--Frobenius theorem, the
Perron vector of $K$ is positive and, by symmetry, constant on each part. 
Let $x_i$ be
its value on the part $V_i$ and $S:=\sum_j n_jx_j$. 
A vertex of part $V_i$ is adjacent
to all vertices outside part $V_i$, so $\lambda x_i= \sum_{j\neq i} n_jx_j = S-n_ix_i$, i.e., 
$$ x_i= \frac{S}{\lambda+n_i} >0. $$ Multiplying by $n_i$, summing, and dividing by $S>0$
gives the characteristic equation
\begin{equation}
        \sum_{i=1}^r\frac{n_i}{\lambda+n_i}=1.
        \label{eq:root-equation}
\end{equation}
Put $m:=n/r$ and $\lambda :=\alpha n+z$ for some $z$, which will be determined later. 
Since $K$ is $K_{r+1}$-free, Wilf's bound gives
$\lambda \le \alpha n$. On the other hand,
$\lambda \ge \frac{2e(K)}{n} 
=\frac{1}{n}\bigl(n^2-\sum_i n_i^2\bigr)
=\alpha n-\frac{B}{n}\ge \alpha n-1$
for large $n$, since $B=O_r(\sqrt n)$. Hence the value
$z=\lambda -\alpha n$ lies in $[-1,0]$,
and it is therefore the unique root of $F$ there. 
Since $\alpha n+m=n$, we have
$\lambda+n_i=n+z+a_i$, so \eqref{eq:root-equation} reads $F(z)=0$, where
\[
        F(z):=\sum_{i=1}^r\frac{m+a_i}{n+z+a_i}-1 .
\]
Since $\lambda$ is the largest root of
\eqref{eq:root-equation}, the corresponding $z$ is the root of $F$.

From $a_i^2\le B=O_r(\sqrt n)$ we get
$|a_i|\le\sqrt B=O_r(n^{1/4})$. Hence for $|z|\le 1$ and all $i$, we have 
$|z+a_i|=O_r(n^{1/4})$, so $(n+z+a_i)^2=n^2\bigl(1+O_r(n^{-3/4})\bigr)$. 
Since $\sum_i(m+a_i)=n$, we get 
\begin{equation}
\label{eq-F-qiudao}
        F'(z)=-\sum_{i=1}^r\frac{m+a_i}{(n+z+a_i)^2}
             =-\frac1{n^2}\sum_{i=1}^r(m+a_i)\bigl(1+O_r(n^{-3/4})\bigr)
             =-\frac1n+O_r(n^{-7/4}). 
\end{equation}
In particular $F'(z)\le-\tfrac1{2n}$ on $z\in [-1,1]$ for large $n$, so $F$ has a unique
root there.

 Using $\tfrac1r+\alpha=1$, we have $m+a_i = \tfrac1r(n+a_i)+\alpha a_i$. Then 
 $\frac{m+a_i}{n+a_i} 
        =\frac1r+\frac{\alpha a_i}{n+a_i}$. So 
\[ 
        F(0)=\alpha\sum_{i=1}^r\frac{a_i}{n+a_i}.
\]
From $\frac{a_i}{n+a_i}=\frac{a_i}{n}-\frac{a_i^2}{n^2}
+O_r\!\bigl(\tfrac{|a_i|^3}{n^3}\bigr)$ and $\sum_i a_i=0$, we get 
\[
        \sum_{i=1}^r\frac{a_i}{n+a_i}
        =-\frac{B}{n^2}+O_r\!\Bigl(\frac{\sum_i|a_i|^3}{n^3}\Bigr)
        =-\frac{B}{n^2}+O_r(n^{-9/4}),
\]
where $\sum_i|a_i|^3\le(\max_i|a_i|)\,B=O_r(n^{1/4})\cdot O_r(\sqrt n)
=O_r(n^{3/4})$. Therefore
\begin{equation}
        F(0)=-\frac{\alpha B}{n^2}+O_r(n^{-9/4}).
        \label{eq:F0}
\end{equation}

 By \eqref{eq:F0}, we have $|F(0)|=O_r(n^{-3/2})$, so with
$F'\le-\tfrac1{2n}$ the root satisfies $|z|\le 2n|F(0)|=O_r(n^{-1/2})<1$. By the
mean value theorem, we get 
$F(z)-F(0)= F'(\zeta) \cdot z$  for some $\zeta$ between $0$ and $z$. Thus, we get  
$z=- \frac{F(0)}{F'(\zeta)}$. 
Inserting \eqref{eq-F-qiudao} and \eqref{eq:F0},
\[
        z=-\frac{-\alpha B/n^2+O_r(n^{-9/4})}{-\tfrac1n+O_r(n^{-7/4})}
         =-\frac{\alpha B}{n}+O_r(n^{-5/4}).
\]
Hence, we obtain 
$    \lambda(K)=\alpha n+z=\alpha n-\frac{\alpha B}{n}+O_r(n^{-5/4})$, 
which proves \eqref{eq:multipartite-expansion}.

\medskip
We write the part sizes of the Tur\'{a}n graph $T_{n,r}$ as 
$t_i=n/r +b_i$ for $i\in \{1,\ldots ,r\}$, where $\sum_i b_i=0$. Since $T_{n,r}$ is balanced, i.e., $|t_i-t_j| \le 1$ for all $i\neq j$, we have $|b_i|\le 1$ and
$$B_T:=\sum_{i=1}^r b_i^2=O_r(1). $$ 
Applying \eqref{eq:multipartite-expansion}
to $T_{n,r}$ gives 
\begin{equation}
        \lambda(T_{n,r})=\alpha n-\frac{\alpha B_T}{n}+O_r(n^{-5/4}).
        \label{eq-expand-lam-T}
\end{equation}
For any complete $r$-partite graph with part sizes $x_1,\ldots ,x_r$ summing to $n$, we have  
$$ e(K_r(x_1,\ldots ,x_r))=\frac12 
\left(n^2-\sum_{i=1}^r x_i^2\right).$$
Since 
$\sum_i n_i^2= \sum_i (\frac{n}{r} +a_i)^2 = \tfrac{n^2}{r}+B$ and $\sum_i t_i^2=\tfrac{n^2}{r}+B_T$, we obtain 
\begin{equation} \label{eq-DB}
        D=e(T_{n,r})-e(K)=\frac12\left(\sum_i n_i^2-\sum_i t_i^2\right)
        =\frac12(B-B_T).
\end{equation}
Subtracting \eqref{eq-expand-lam-T} from \eqref{eq:multipartite-expansion}, it follows that 
\[
        \lambda(K)=\lambda(T_{n,r})-\frac{\alpha(B-B_T)}{n}+O_r(n^{-5/4})
                 =\lambda(T_{n,r})-\frac{2\alpha D}{n}+O_r(n^{-5/4}). \qedhere 
\]
\end{proof}

\subsection{The spectral radius of Tur\'{a}n's graph plus a matching}

The graph $Y_{n,r,q}$ is obtained from the Tur\'{a}n graph $T_{n,r}$ by adding a matching with $q$ edges into a largest part. We show that such a matching raises the spectral radius by $2q/n+O_r(n^{-3/2})$.

\begin{lemma} 
\label{lem:Y-expansion}
If $q$ is a positive integer with $q=O_r(\sqrt n \, )$, then
\begin{equation*}
        \lambda(Y_{n,r,q})
        =\lambda(T_{n,r})+\frac{2q}{n}+O_r(n^{-3/2}). 
\end{equation*}
\end{lemma}

\begin{proof}
Let $A_T$ be the adjacency matrix of the Tur\'{a}n graph $T_{n,r}$. Let $M$ be the adjacency matrix of the
$q$ disjoint edges added to $T_{n,r}$ to form $Y_{n,r,q}$. Then
\[
        A(Y_{n,r,q})=A_T+M.
\]
Let $  \mu:=\lambda(T_{n,r})$ 
and let $\bm{u}\in \mathbb{R}^n$ be a unit Perron vector of the Tur\'{a}n graph $T_{n,r}$. 
Because the part sizes of $T_{n,r}$ differ by at most one, we first show that for every vertex $v$,
\begin{equation}
        u_v=\frac1{\sqrt n}+O_r(n^{-3/2}).
        \label{eq-Perron-vect-uniform-T}
\end{equation} 
By the Perron--Frobenius theorem and the symmetry of $T_{n,r}$, the coordinate  $u_v$ is constant on
each part. We write $u_v=c_i$ for $v$ in part $i$. The eigenvalue equation at a
vertex of part $i$ reads $\mu c_i=\sum_{j\ne i}t_jc_j=C-t_ic_i$, where
$C:=\sum_j t_jc_j$. Hence, we get  
\begin{equation} \label{eq:coordinate}
c_i= \frac{C}{\mu+t_i}. 
\end{equation}
The part sizes of $T$ differ by at most one and satisfy $t_i=n/r+O_r(1)$ for all $i$. Since $\delta (T_{n,r})\le \mu \le \Delta (T_{n,r})$, we have  $\mu=\alpha n+O_r(1)$ and $ \mu+t_i=n+O_r(1)$. Then  
$  \frac{1}{\mu+t_i}=\frac{1}{n}\bigl(1+O_r(n^{-1})\bigr)$. 
Using normalization and $\sum_i t_i=n$ gives
\[
        1=\sum_v u_v^2=\sum_{i=1}^r t_ic_i^2=\sum_{i=1}^r\frac{t_iC^2}{(\mu+t_i)^2}
         =\frac{C^2}{n^2}\left(\sum_{i=1}^r t_i\right)\bigl(1+O_r(n^{-1})\bigr)
         =\frac{C^2}{n}\bigl(1+O_r(n^{-1})\bigr).
\]
Therefore $C^2=n+O_r(1)$, i.e.,  
$C=\sqrt n+O_r(n^{-1/2})$. Then 
\[
        c_i=\frac{C}{\mu+t_i}
           =\frac{\sqrt n+O_r(n^{-1/2})}{n+O_r(1)} 
           =\frac{1}{\sqrt n}+O_r(n^{-3/2}),
\]
which proves the estimate in \eqref{eq-Perron-vect-uniform-T}. 

First we prove the lower bound on $\lambda (Y_{n,r,q})$. 
By the Rayleigh quotient,
\[
        \lambda(Y_{n,r,q})
        \ge \bm{u}^T(A_T+M) \bm{u}
        =\mu+ \bm{u}^TM \bm{u}.
\]
Each added matching edge $xy$ contributes $2u_xu_y$ to $\bm{u}^TM \bm{u}$. By
\eqref{eq-Perron-vect-uniform-T},
\[
        2u_xu_y=\frac2n+O_r(n^{-2}).
\]
There are $q$ added edges, so
\begin{equation}
        \bm{u}^TM\bm{u}=\frac{2q}{n}+O_r\left(\frac q{n^2}\right)
             =\frac{2q}{n}+O_r(n^{-3/2}).
        \label{eq:uMu-matching}
\end{equation}
Thus
\begin{equation}
        \lambda(Y_{n,r,q})
        \ge \mu+\frac{2q}{n}+O_r(n^{-3/2}).
        \label{eq:Y-lower}
\end{equation} 
Now we prove the upper bound on $\lambda(Y_{n,r,q})$. 
Let $\bm{x} \in \mathbb{R}^n$ be any unit vector. Write
\[
        \bm{x}=a \bm{u}+ \bm{y},
        \qquad
        \bm{y}\perp \bm{u},
        \qquad
        a^2+\|\bm{y}\|^2=1.
\]
Put $\beta:=\|\bm{y}\|$. 
Since $T_{n,r}$ is a complete $r$-partite graph, $T_{n,r}$ has exactly one positive eigenvalue. Therefore
$ \bm{y}^TA_T \bm{y}\le 0$. 
It follows that $ \bm{x}^TA_T \bm{x}\le \mu a^2$. 
Also,
\[
        \bm{x}^TM \bm{x} 
        =a^2\bm{u}^TM \bm{u} 
        +2a\,\bm{y}^TM \bm{u} + \bm{y}^TM\bm{y}.
\]
Since $M$ is the adjacency matrix corresponding to the added matching, we get $\|M\|=1$. Hence
\[
        \bm{y}^TM\bm{y}\le \|M\|\| \bm{y} \|^2=\beta^2.
\]
Using $a^2=1-\beta^2$ and $|a|\le 1$, we obtain
\begin{align*}
        \bm{x}^T(A_T+M)\bm{x}
        &\le \mu a^2+a^2 \bm{u}^TM \bm{u}+2|a|\,|\bm{y}^TM\bm{u}|+\beta^2  \\
        &\le \mu+ \bm{u}^TM \bm{u}-(\mu+ \bm{u}^TM\bm{u}-1)\beta^2+2\beta\|M\bm{u}\|.
\end{align*}
Since $q=O_r(\sqrt{n})$, we have  $\bm{u}^TM\bm{u} =o(1)$ by (\ref{eq:uMu-matching}). 
For sufficiently large $n$, we get 
\[
        \mu+ \bm{u}^TM \bm{u} -1=\alpha n+o(n)>0.
\]
For $L>0$, the maximum of $-L\beta^2+2B\beta$ over $\beta\ge 0$ is $B^2/L$. Applying this
with
\[
        L:=\mu+ \bm{u}^TM \bm{u}-1,
        \qquad
        B:=\|M \bm{u}\|,
\]
we get
\begin{equation}
        \lambda(Y_{n,r,q})
        \le \mu+ \bm{u}^TM\bm{u}+
             \frac{\|M \bm{u}\|^2}{\mu+ \bm{u}^TM\bm{u}-1}.
        \label{eq:Y-upper-raw}
\end{equation}
It remains to estimate the last term. Because $M$ is a matching, $M \bm{u}$ is supported on the
$2q$ endpoints of the matching, and each nonzero coordinate is
$1/\sqrt n+O_r(n^{-3/2})$. Therefore
\[
        \|M\bm{u}\|^2
        =2q\left(\frac1n+O_r(n^{-2})\right)
        =O_r\left(\frac qn\right)
        =O_r(n^{-1/2}).
\]
The denominator in \eqref{eq:Y-upper-raw} is $\alpha n+o(n)$, so
\[
        \frac{\|M \bm{u}\|^2}{\mu+ 
        \bm{u}^TM\bm{u} -1}
        =O_r(n^{-3/2}).
\]
Combining this with \eqref{eq:uMu-matching} and \eqref{eq:Y-upper-raw} gives
\[
        \lambda(Y_{n,r,q})
        \le \mu+\frac{2q}{n}+O_r(n^{-3/2}).
\]
Together with the lower bound \eqref{eq:Y-lower}, this proves the lemma.
\end{proof}

\subsection{A sharp construction by adding a star}

To show the tightness of $\delta_r$ in Theorem \ref{thm-cliques}, we need to construct extremal examples with spectral radius at least $\lambda (Y_{n,r,q})$ for infinitely many $n$. 
Throughout this section, we assume that $ n=r m$. 

\begin{proposition} 
\label{prop:construction} 
    Assume that $r\ge 2$ and $\alpha = \frac{r-1}{r}$. 
    Suppose that $n=rm$ is sufficiently large  and $q=O_r(\sqrt{n})$. 
    Let $H=H_{n,r,q}$ be the graph obtained from the complete $r$-partite graph $ K_0:=K_r(m+1,m-1,m,\ldots,m)$ by adding a $q$-edge star into the partite set of size $m+1$. Then 
    \[   \lambda(H)
        \ge \alpha n-\frac{2\alpha}{n}+\frac{2q}{n}
              +\frac{q^2\alpha^{-1}}{n^2}+O_r(n^{-3/2}).  \]
\end{proposition}

By Theorem \ref{thm:multipartite-expansion}, we can get $\lambda (K_0) \approx \alpha n - \frac{2\alpha}{n}$. 
The remaining argument is to prove that adding a star with $q$ edges increases $\lambda (K_0)$ by about  $ \frac{2q}{n}  + \frac{q^2\alpha^{-1}}{n^2}$. 
Here, the second order term is  $\frac{q^2\alpha^{-1}}{n^2}$, 
which is different from that of adding a matching. 
To determine this term, we use a test
vector spanned by $\bm u$ and a second vector $\bm y$, which we get by centering
the star-degree function on the part that contains the star. This $\bm y$ is orthogonal to
$\bm u$ and is sent to zero by $A(K_0)$, but $\bm y^{T}S\bm u$ is large because of the star;
choosing the best vector in this plane then gives this missing term exactly.

\begin{proof} 
First, we apply Theorem \ref{thm:multipartite-expansion} to estimate the spectral radius of $K_0$.  
Comparing to the balanced $K_r(m,m,\ldots,m)$, the deviations are
$ 1,-1,0,\ldots,0$. 
Thus $  B=1^2+(-1)^2=2$.  
Since the deviations of $K_0$ are bounded ($\max_i|a_i|=1$ and $B=2$), the
estimates in the proof of Theorem \ref{thm:multipartite-expansion} sharpen to
$F(0)=-\tfrac{\alpha B}{n^2}+O_r(n^{-3})$ and $F'(z)=-\tfrac1n+O_r(n^{-2})$. 
Hence
\begin{equation*}
        \lambda(K_0)=\alpha n-\frac{2\alpha}{n}+O_r(n^{-2}).
\end{equation*}  
Let $S$ be the adjacency matrix of the added $q$-edge star.  Let $\bm{u}$ be the unit
Perron vector of $K_0$. Since the parts of $K_0$ differ from $n/r$ by only $O_r(1)$, as in (\ref{eq-Perron-vect-uniform-T}), the coordinates of the Perron vector $\bm{u}$ are almost equal, and they are given as  
\begin{equation}
        u_0 =\frac1{\sqrt n}+O_r(n^{-3/2}).
        \label{eq:u0-star}
\end{equation}
The first-order term of the added star is
\begin{equation}
        \bm{u}^TS\bm{u}=2q u_0^2=\frac{2q}{n}+O_r\left(\frac q{n^2}\right)
              =\frac{2q}{n}+O_r(n^{-3/2}).
        \label{eq:uSu-star}
\end{equation} 
From the above discussion, the Rayleigh quotient implies 
$\lambda (H)\ge  \alpha n - {2\alpha}/{n} + 2q/n + O_r(n^{-3/2})$. 
Thus, the difficulty is to show the second-order term $q^2\alpha^{-1}/n^2$ of the increment caused by the added star.  
To achieve this, we need to construct a new test vector instead of $\bm{u}$.

Let $U$ be the partite set in $K_0$ of size $|U|=m+1$,
let $c$ be the center vertex of the star, and let $L$ be the set of its $q$ leaves. 
Define the star-degree function $d$ on $U$ by $ d(c)=q$, $ d(x)=1 $ for $x\in L$, 
and $d(x)=0$ for $x\in U\setminus(\{c\}\cup L)$.  
Then $   \sum_{x\in U}d(x)=2q$. 
Let
\[
        \bar d:=\frac{2q}{|U|}=\frac{2q}{m+1}.
\]
Define a vector $\bm{w} \in \mathbb{R}^n$ by
\[
        w_x=\begin{cases}
             d(x)-\bar d, &\text{if}~ x\in U;\\
             0,           &\text{if}~ x\notin U.
             \end{cases}
\]
Then
\[
        \sum_{x\in U}w_x
        =\sum_{x\in U}d(x)-|U|\bar d
        =2q-(m+1)\frac{2q}{m+1}=0.
\]
Because $\bm{u}$ is constant restricted to $U$ and $\bm{w}$ is supported only on $U$, this gives $  \bm{w}\perp \bm{u}$. 
Since $\bm{w}$ is supported on one part and has coordinate sum zero within that part, the 
adjacency matrix $A(K_0)$ sends $\bm{w}$ to zero, that is,  $A(K_0)\bm{w}=0$.  
We next compute $\|\bm{w}\|^2$. Since $w_x=d(x)-\bar d$ on $U$,
\begin{align*}
        \|\bm{w}\|^2
        =\sum_{x\in U}(d(x)-\bar d)^2  
        =\sum_{x\in U}d(x)^2
          -2\bar d\sum_{x\in U}d(x)
          +|U|\bar d^2.
\end{align*}
Now
\[
        \sum_{x\in U}d(x)^2=q^2+q,
\]
while
\[
        2\bar d\sum_{x\in U}d(x)
        =2\cdot\frac{2q}{m+1}\cdot 2q
        =\frac{8q^2}{m+1},
\]
and
\[
        |U|\bar d^2
        =(m+1)\frac{4q^2}{(m+1)^2}
        =\frac{4q^2}{m+1}.
\]
Since $q=O(\sqrt n)$, we get 
\begin{equation}
        \|\bm{w}\|^2=q^2+q-\frac{4q^2}{m+1}=q^2+O(q).
        \label{eq:w-norm}
\end{equation} 
Let $ \bm{y}:=\frac{\bm{w}}{\| \bm{w}\|}$. 
Since $\bm{w}\perp \bm{u}$ and $A(K_0)\bm{w}=0$, we get 
\begin{equation}
        \|\bm{y}\|=1,
        \qquad
        \bm{y}\perp \bm{u},
        \qquad
        A(K_0) \bm{y} =0.
        \label{eq:y-properties}
\end{equation}

We compute $\bm{y}^TS \bm{u}$. Since $S$ is the adjacency matrix of the star and $u_x=u_0$ for all
$x\in U$, we have 
$ (S \bm{u})_x=u_0 \,d(x)$ if 
$x\in U$; 
and $(S \bm{u})_x=0$ if $x\notin U$. Hence
\[
        \bm{w}^TS \bm{u}=u_0\sum_{x\in U}w_xd(x).
\]
But
\begin{align*}
        \sum_{x\in U}w_xd(x)
          =\sum_{x\in U}d(x)^2-\bar d\sum_{x\in U}d(x)  
        =(q^2+q)-\frac{2q}{m+1}\cdot 2q 
        =q^2+q-\frac{4q^2}{m+1}.
\end{align*}
Therefore
$ \bm{w}^TS \bm{u}=u_0\| \bm{w} \|^2$ by \eqref{eq:w-norm}, 
and so
\begin{equation}
        \bm{y}^TS \bm{u}=u_0\| \bm{w} \|.
        \label{eq:ySu}
\end{equation} 
Let $  \mu:=\lambda(K_0)$ and $ t:={\bm{y}^TS \bm{u}}/{\mu}$. 
Then the vector
\[
        \bm{x}:=\frac{\bm{u}+t \bm{y}}{\sqrt{1+t^2}}
\]
has norm one. Therefore, by the Rayleigh quotient,
\begin{align} 
        \lambda(H)
        =\lambda(K_0+S)  
        \ge \bm{x}^T(A(K_0)+S) \bm{x}  =\frac{\mu+ \bm{u}^TS \bm{u}+2t\,\bm{y}^TS\bm{u} 
        +t^2\bm{y}^TS\bm{y}}{1+t^2},
        \label{eq:rayleigh-star}
\end{align}
where we used \eqref{eq:y-properties} to remove all cross terms involving $A(K_0)$.

Next, by \eqref{eq:u0-star}, \eqref{eq:w-norm}  and \eqref{eq:ySu},
\begin{align*}
        (\bm{y}^TS\bm{u})^2
        =u_0^2\|\bm{w}\|^2 =\left(\frac1n+O_r(n^{-2})\right)(q^2+O(q)) 
        =\frac{q^2}{n}+O_r(n^{-1/2}). 
\end{align*} 
Since $\mu=\alpha n+O_r(n^{-1})$, we get
\begin{align}
        \frac{(\bm{y}^TS\bm{u})^2}{\mu}
        =\frac{q^2/n+O_r(n^{-1/2})}{\alpha n+O_r(n^{-1})}  
        =\frac1\alpha\frac{q^2}{n^2}+O_r(n^{-3/2}) .
        \label{eq:star-second-order}
\end{align} 
We now simplify the right-hand side of \eqref{eq:rayleigh-star}. The added star is a
copy of $K_{1,q}$, so its adjacency matrix has spectral norm
$ \|S\|=\sqrt q=O_r(n^{1/4})$, 
and hence $|\bm{y}^TS \bm{y}|\le\|S\|\,\|\bm{y}\|^2=O_r(n^{1/4})$. By \eqref{eq:ySu},
\eqref{eq:u0-star} and \eqref{eq:w-norm},
we have $\bm{y}^TS \bm{u}=u_0\| \bm{w}\|=O_r(1)$. Since
$\mu=\alpha n+O_r(n^{-1})$,
we get $t={\bm{y}^TS\bm{u}}/{\mu}=O_r(n^{-1})$. 
Note that $t\,\bm{y}^TS\bm{u}=(\bm{y}^TS\bm{u})^2/\mu=\mu t^2$. So the numerator in \eqref{eq:rayleigh-star} regroups as
\[
        \mu+\bm{u}^TS\bm{u}+2t\,\bm{y}^TS\bm{u}+t^2\bm{y}^TS\bm{y}
        =\mu(1+t^2)+\bm{u}^TS\bm{u}+\frac{(\bm{y}^TS\bm{u})^2}{\mu}+t^2\bm{y}^TS\bm{y} .
\]
Dividing by $1+t^2$, we find that (\ref{eq:rayleigh-star}) becomes 
\begin{equation} \label{eq:after-regroup}
        \frac{\mu+\bm{u}^TS\bm{u}+2t\,\bm{y}^TS\bm{u}+t^2\bm{y}^TS\bm{y}}{1+t^2}
        =\mu+\frac{1}{1+t^2}\left(\bm{u}^TS\bm{u}+\frac{(\bm{y}^TS\bm{u})^2}{\mu}+t^2\bm{y}^TS\bm{y}\right).
\end{equation}
By \eqref{eq:uSu-star} and \eqref{eq:star-second-order}, we get 
$\bm{u}^TS\bm{u}=O_r(n^{-1/2})$ and $(\bm{y}^TS\bm{u})^2/\mu=O_r(n^{-1})$, while
\[
        |t^2\bm{y}^TS\bm{y}|\le t^2\|S\|=O_r(n^{-2})\cdot O_r(n^{1/4})=O_r(n^{-7/4}).
\]
Hence the parenthesis equals $\bm{u}^TS\bm{u}+(\bm{y}^TS\bm{u})^2/\mu+O_r(n^{-7/4})=O_r(n^{-1/2})$, and
multiplying by $\tfrac1{1+t^2}=1+O_r(n^{-2})$ alters it by only
$O_r(n^{-1/2})\cdot O_r(n^{-2})=O_r(n^{-5/2})$. Therefore (\ref{eq:after-regroup}) gives 
\begin{equation*}
        \lambda(H)
        \ge\mu+\bm{u}^TS\bm{u}+\frac{(\bm{y}^TS\bm{u})^2}{\mu}+O_r(n^{-7/4}).
\end{equation*} 
Using \eqref{eq:uSu-star} and \eqref{eq:star-second-order}, we obtain 
\[
        \lambda(H)
        \ge \mu+\frac{2q}{n}+\frac{q^2\alpha^{-1}}{ n^2}+O_r(n^{-3/2}).
\]
Recalling that $\mu = \lambda (K_0)= \alpha n- {2\alpha}/{n} + O_r(n^{-2})$, we 
obtain the desired bound.  
\end{proof}

\section{Spectral estimates for nearly Tur\'{a}n graphs}

\label{sec-1-8-and-1-7}

In this section, we first prove the second-order  spectral estimate in Theorem \ref{thm:local-spectral}. Second, we prove the continuous supersaturation in Theorem \ref{thm-continuous} by using Theorem \ref{thm:local-spectral} and Theorem \ref{thm-Y}.

\subsection{Proof of Theorem \ref{thm:local-spectral}}

Our goal is to bound the Rayleigh quotient $\bm{x}^TA(G)\bm{x}$ for an arbitrary unit vector $\bm{x}$. 
We write $A(G):=A(K)+E$, where $E$ has a $+1$ for each added class-edge and a
$-1$ for each deleted cross-edge.  
 Note that $A(K)$ has only one positive eigenvalue $\mu$ with the unit Perron vector $\bm{u}$.  
We denote  
$\bm{x}:=a\bm u+\bm y$ for some $\bm{y}$ with $\bm y\perp\bm u$, then bound the part from $A(K)$ by $\mu a^2$ and split the rest into a first-order term $e_0:=\bm u^TE\bm u$ and a second-order term coming from
$\|E\bm u\|^2/\mu$. The first term gives $2(s-h)/n$, and the second, after we bound $\|E\bm u\|^2$
by the degree-square sum (Lemma \ref{lem-deg-squ}), gives $(s^2+h^2)\alpha^{-1}/n^2$, where the
coefficient $\alpha^{-1}$ comes from the spectral radius $\mu\approx\alpha n$. 

\begin{proof}[Proof of Theorem \ref{thm:local-spectral}]
As stated above, 
we denote $A(G):=A_0+E$, where $A_0=A(K)$ and $E$ is a symmetric matrix with entries $+1$ corresponding to the
added class-edges and $-1$ corresponding to the deleted cross-edges. Let $\mu=\lambda(K)$ and let 
$\bm{u}$ be the positive unit Perron vector of $K$. 
By the identity \eqref{eq-DB} and $B_T=O_r(1)$, the hypothesis $D=O_r(\sqrt n)$ gives
$$B:=\sum_{i=1}^r \Big(n_i-\frac nr \Big)^2=O_r(\sqrt n),$$
 and then $\max_{i\in [r]} |n_i-\tfrac nr|=O_r(n^{1/4})$.
By Theorem \ref{thm:multipartite-expansion}, we have 
$$ \mu=\alpha n-\frac{\alpha B}{n}+O_r(n^{-5/4}), $$ 
which together with $B=O_r(\sqrt n)$ yields  
\begin{equation} \label{eq:mu}
\mu=\alpha n+O_r(n^{-1/2}). 
\end{equation}
Similar to (\ref{eq:coordinate}), 
we obtain 
$u_v= \frac{C}{\mu+n_i }$ for every $v\in V_i$ with $C>0$. Since $n_i=\frac{n}{r} + O_r(n^{1/4})$, 
we have $\mu+n_i= n+O_r(n^{1/4})$ and 
$   u_v=\frac{C}{\mu+n_i}=\frac{C}{n}\bigl(1+O_r(n^{-3/4})\bigr)$. 
 Then $\sum_v u_v^2=1$ forces 
 $$C=\sqrt n\,\bigl(1+O_r(n^{-3/4})\bigr).$$ 
Then for every $v\in V$, we have 
\begin{equation}
        u_v=\frac1{\sqrt n}\bigl(1+O_r(n^{-3/4})\bigr) .
        \label{eq:u-uniform-local}
\end{equation} 
Put $e_0:=\bm{u}^{T}E\bm{u}$. An added class-edge contributes $2u_xu_y$ and a deleted
cross-edge $-2u_xu_y$. By \eqref{eq:u-uniform-local}, we have 
$2u_xu_y=\frac2n+O_r(n^{-7/4})$. As there are $s$ added and $h$ deleted edges and
$s+h=O_r(\sqrt n)$, we conclude that 
\begin{equation} \label{eq-e0}
        e_0=\frac{2(s-h)}{n}+O_r\bigl((s+h)n^{-7/4}\bigr)
           =\frac{2(s-h)}{n}+O_r(n^{-5/4}).
\end{equation} 
By the definition of $E$, it follows that 
$$(E\bm{u})_v=\sum_{vw\ \mathrm{added}}u_w-\sum_{vw\ \mathrm{deleted}}u_w.$$ 
For a vertex $v$, let $d^{+}(v)$ and $d^{-}(v)$ be its degrees in the graphs formed by
the added class-edges and by the deleted cross-edges, respectively. 
By \eqref{eq:u-uniform-local}, we obtain 
\[
        (E\bm{u})_v=\frac1{\sqrt n}\bigl(d^{+}(v)-d^{-}(v)\bigr)
              +O_r\bigl(n^{-5/4}(d^{+}(v)+d^{-}(v))\bigr).
\]
We now bound $\|E\bm{u}\|^2$. 
Set $X:=\sum_v\bigl(d^{+}(v)^2+d^{-}(v)^2\bigr)$. Squaring and summing the above over all $v$, the main term
is at most $\frac1n\sum_v(d^{+}(v)-d^{-}(v))^2\le\frac1nX$, while
the squared error term is $O_r(n^{-5/2})X$, and 
the cross term is 
$O_r(n^{-7/4})X$.  Hence $\|E\bm{u}\|^2\le\frac{1}{n}X + O_r(n^{-7/4})X$. 

 Applying Lemma \ref{lem-deg-squ} 
to the graphs consisting of the added class-edges and deleted cross-edges, respectively, we get  $X\le(s^2+s)+(h^2+h)$. Therefore
\begin{equation} \label{eq:Eu}
        \|E\bm{u}\|^2\le\frac{s^2+h^2}{n}+O_r(n^{-1/2}).
\end{equation}
Let $\bm{x}$ be an arbitrary unit vector. We write 
$$\bm{x} :=a\bm{u}+\bm{y}$$
 with some vector $\bm{y}$ satisfying $\bm{y}\perp \bm{u}$, so that $a^2+\beta^2=1$,
where $\beta:=\|\bm{y}\|$. Every eigenvalue of $K$ other than
$\mu$ is nonpositive, so $\bm{y}^{T}A_0\bm{y}\le0$ and $\bm{x}^{T}A_0\bm{x}\le\mu a^2$. Moreover
$\|E\|\le \sqrt{2(s+h)}=O_r(n^{1/4})$. The norm bound gives  $\bm{y}^{T}E\bm{y}\le\|E\|\beta^2$, 
and the Cauchy--Schwarz inequality gives $|\bm{y}^{T}E\bm{u}|\le\beta\|E\bm{u}\|$. 
Using $a^2=1-\beta^2$ and $|a|\le1$, expanding $\bm{x}^TA(G)\bm{x}$ gives 
\[
   \lambda (G)\le \bm{x}^{T}A(G)\bm{x}\le \mu+e_0-(\mu+e_0-\|E\|)\beta^2+2\beta\|E\bm{u}\|.
\]
Combining with (\ref{eq:mu}) and (\ref{eq-e0}), 
we have $\mu+e_0-\|E\|=\alpha n+O_r(n^{1/4})>0$. 
The maximizer 
$\beta=\|E\bm{u}\|/(\mu+e_0-\|E\|)=O_r(n^{-3/4})$ lies in $[0,1]$. Maximizing over
$\beta\ge0$ gives
\begin{equation} \label{eq:upper}
        \lambda(G)\le \mu+e_0+\frac{\|E\bm{u}\|^2}{\mu+e_0-\|E\|}.
\end{equation}
 Since
$\mu+e_0-\|E\|=\alpha n\bigl(1+O_r(n^{-3/4})\bigr)$, using (\ref{eq:Eu}) gives  
\[
        \frac{\|E\bm{u}\|^2}{\mu+e_0-\|E\|}
        =\frac{(s^2+h^2)/n+O_r(n^{-1/2})}{\alpha n\bigl(1+O_r(n^{-3/4})\bigr)}
        =\frac{(s^2+h^2)\alpha^{-1}}{ n^2}+O_r(n^{-3/2}).
\]
By Theorem \ref{thm:multipartite-expansion} again, we have 
\[ \mu= \lambda (T_{n,r}) - \frac{2\alpha D}{n} + O(n^{-5/4}).  \]
Substituting $\mu$, $e_0$ and $\|E\bm{u}\|^2/(\mu+e_0-\|E\|)$ into
(\ref{eq:upper}) gives the desired  bound of $\lambda (G)$.
\end{proof}

\subsection{Continuous supersaturation: Proof of Theorem \ref{thm-continuous}} 

Next, we provide an immediate application of Theorem \ref{thm:local-spectral}, showing that the spectral supersaturation in Theorem \ref{thm-continuous} is a direct consequence of Theorem \ref{thm-Y} combined with Theorem \ref{thm:local-spectral}.

\begin{proof}[Proof of Theorem \ref{thm-continuous}]
Let $\delta_F>0$ be the constant 
in Theorem~\ref{thm-Y}, and write 
$\alpha:=1-\tfrac1r$. Adding $q$ disjoint edges to 
$T_{n,r}$ raises the spectral radius by $2q/n$ up to 
lower-order terms. By Theorem~\ref{thm:local-spectral}, for every
$q=O(\sqrt n)$, we have 
\begin{equation} \label{eq-Y-upper}
    \lambda(Y_{n,r,q})
    \le \lambda(T_{n,r})+\frac{2q}{n}
        +\frac{\alpha^{-1}q^2}{n^2}+O(n^{-5/4}).
\end{equation} 
Fix $C$ with $0<C\le\delta_F/\sqrt n$, and set
$q:=\lfloor \frac{1}{2}Cn\rfloor-K$, where $K=K(r,F)$ is a constant integer 
fixed below. Then
$  q\le\frac{1}{2}Cn \le\frac{1}{2} \delta_F\sqrt n$. 
So $q=O(\sqrt n)$ and (\ref{eq-Y-upper}) applies. Moreover, the two
error terms in (\ref{eq-Y-upper}) are at most
$\frac{1}{4n} \alpha^{-1}\delta_F^2 +O(n^{-5/4})\le K'/n$ for
some absolute constant $K'$. Hence
\[
    \lambda(Y_{n,r,q})
    \le\lambda(T_{n,r})
       +\frac{2}{n}\Bigl(\frac{Cn}{2}-K\Bigr)+\frac{K'}{n}
    =\lambda(T_{n,r})+C-\frac{2K-K'}{n}.
\]
Choosing $K\ge K'/2$ gives
$\lambda(Y_{n,r,q})\le\lambda(T_{n,r})+C\le\lambda(G)$.

Set $\eps_n:=-(K+1)/n$. If $q\ge1$, then
$1\le q\le\delta_F\sqrt n$, so applying Theorem~\ref{thm-Y}
gives $\texttt{\#}F(G)\ge q\,c(n,F)$. Combining with 
$q\ge\frac{Cn}{2}-(K+1)$, we get 
$ \texttt{\#}F(G)\ge \big(\frac{C}{2}+\eps_n\big)
        n\,c(n,F)$.  
If $q<1$, as $K$ is an integer, then $\frac{Cn}{2}<K+1$, so
$\bigl(\frac C2+\eps_n\bigr)n=\frac{Cn}{2}-(K+1)<0$ and the desired 
inequality holds trivially. In either case, we have 
$\texttt{\#}F(G)\ge\bigl(\frac C2 - \frac{K+1}{n}\bigr)n\,c(n,F)=(\frac{C}{2} - o(1))n\, c(n,F)$, as needed. 
\end{proof}

\section{The sharp threshold for cliques} 

\label{sec:clique} 

In this section, we prove Theorem~\ref{thm-cliques}. 
To start with, we show the tightness of $\delta_r= (1-\frac{1}{r})\sqrt{2}$.

\begin{example}  \label{exam-clique}
For any fixed $\varepsilon>0$ and infinitely many $n$, 
let $q$ be an integer with $(\delta_r +\eps)\sqrt{n} \le q =O(\sqrt{n})$, 
 there exists an $n$-vertex graph $H$ such that 
$  \lambda(H)>\lambda(Y_{n,r,q})$ and 
$   \texttt{\#} K_{r+1}(H)<q\,c(n,K_{r+1})$. 
\end{example}

\begin{proof}
Let $H$ be the graph defined in  Proposition \ref{prop:construction}. 
Each copy of $K_{r+1}$ using an added star edge must choose one vertex from each of the other
$r-1$ parts. These other parts have sizes
$m-1,m,\ldots,m$. 
Therefore, each added star edge lies in exactly $ (m-1)m^{r-2}$ 
copies of $K_{r+1}$. Since a copy of $K_{r+1}$ in an $r$-partite base graph plus internal edges
has exactly one internal edge, these copies are counted once for each star edge. Hence 
\[
     \texttt{\#} K_{r+1}(H)=  q(m-1)m^{r-2}<qm^{r-1}=q\,c(n,K_{r+1}).
\]
By Lemma \ref{lem:Y-expansion}, and since $T_{n,r}$ is exactly balanced when
$n=rm$,
\begin{equation*}
        \lambda(Y_{n,r,q})
        =\alpha n+\frac{2q}{n}+O_r(n^{-3/2}).
\end{equation*}
Applying Proposition \ref{prop:construction}, we obtain
   \[   \lambda(H)
        \ge \alpha n-\frac{2\alpha}{n}+\frac{2q}{n}
              +\frac{q^2\alpha^{-1}}{n^2}+O_r(n^{-3/2}).  \]
    Combining with the above two estimates, we get 
\begin{equation*}
        \lambda(H)-\lambda(Y_{n,r,q})
        \ge -\frac{2\alpha}{n}+\frac{q^2\alpha^{-1}}{n^2}+O_r(n^{-3/2}).
\end{equation*} 
Since $  q =(\delta_r+\varepsilon)\sqrt n$ and $\delta_r=\sqrt{2} (1- \frac{1}{r}) =\sqrt{2} \alpha$, we have 
\[
        \frac{q^2\alpha^{-1}}{n^2}
        = \alpha^{-1}(\delta_r+\varepsilon)^2\frac1n
        > \alpha^{-1} \delta_r^2\frac1n
        =\frac{2\alpha}{n}.
\]
Thus, we obtain $ \lambda(H)>\lambda(Y_{n,r,q})$ for all sufficiently large
$n$.  
\end{proof}

\subsection{Stability and counting for cliques}

Applying Theorem \ref{thm:FLLM}, 
we get the following structural reduction.

\begin{lemma} 
\label{lem:reduction}
Fix $r\ge 2$, $\varepsilon>0$, and let $n$ be sufficiently large. Suppose that
$  1\le q\le (\delta_r -\varepsilon)\sqrt n$ 
and suppose that $G$ is an $n$-vertex graph satisfying
$$  \lambda(G)\ge \lambda(Y_{n,r,q}) $$
and
$$   \texttt{\#} K_{r+1}(G) \le q\,c(n,K_{r+1}). $$
Then there exists 
$  K=K_r(n_1,\ldots,n_r)$ 
such that $G$ is obtained from $K$ by adding $s$ class-edges and deleting $h$ cross-edges, 
where $s,h=O_r(q)$. Moreover, 
$ D:=e(T_{n,r})-e(K) =O_r(q)$. 
\end{lemma}

\begin{proof}
By Lemma~\ref{lem:Y-expansion}, we have 
$\lambda(Y_{n,r,q})=\lambda(T_{n,r})+2q/n
+O_r(n^{-3/2})>\lambda(T_{n,r})$
for all sufficiently large $n$ (using $q\ge1$). Hence, 
the hypothesis gives $\lambda(G)\ge\lambda(T_{n,r})$. 
Since $q\le(\delta_r-\varepsilon)\sqrt n\le\delta_F n$ and $\texttt{\#} K_{r+1}(G)\le q\, c(n,K_{r+1})$,  
applying Theorem~\ref{thm:FLLM} with $F=K_{r+1}$, we obtain a complete 
$r$-partite graph $K=K_r(n_1,\ldots,n_r)$ from which $G$ is built by adding
$s$ class-edges and deleting $h$ cross-edges, where $ h\le 2s = O_r(q)$ and 
$ |n_i-n_j|\le 4\sqrt s $ for all $i\neq j$.  
Write $n_i=n/r+a_i$ with $\sum_i a_i=0$. Since $|n_i-n_j|\le 4\sqrt s$, each $|a_i|\le 4\sqrt s$, so
$B:=\sum_i a_i^2\le 16rs=O_r(s)$. 
By \eqref{eq-DB}, we get 
$D=\tfrac12(B-B_T)$ with $B_T=O_r(1)\ge 0$, hence
$    D\le\tfrac12 B=O_r(s)=O_r(q)$, as needed. 
\end{proof}

The next lemma shows that, in the multipartite setting, each added class-edge
creates about $c(n,K_{r+1})$ copies of $K_{r+1}$, up to losses from imbalance and from
deleted cross-edges.

\begin{lemma}
\label{lem:local-counting}
Let $K=K_r(n_1,\ldots,n_r)$ be an $n$-vertex complete $r$-partite graph, and let $G$ be
obtained from $K$ by adding $s$ class-edges and deleting $h$ cross-edges. Put
$D:=e(T_{n,r})-e(K)$ and assume $s,h,D=O_r(\sqrt n)$. Then there is a constant $C_r>0$,
depending only on $r$, such that
\begin{equation*}
        \texttt{\#} K_{r+1}(G)\ge s\,c(n,K_{r+1})-C_r\,s(\sqrt D+h+1)\,n^{r-2}. 
\end{equation*}
\end{lemma} 

\begin{proof}
Let $V_1,\ldots,V_r$ be the parts of $K$, $|V_i|=n_i$. Call a copy of $K_{r+1}$
\emph{canonical} if it has two vertices in one part and one vertex in each of the others.
Its $r+1$ vertices then form exactly one within-part pair, necessarily a class-edge, every
other pair being a cross-edge. So each canonical copy uses exactly one class-edge, and it is a copy in $G$
precisely when none of its cross-edges has been deleted. Hence, 
summing over the class-edges, we count each canonical copy once.  
We therefore count canonical copies that use only cross-edges of $K$ and subtract
those meeting a deleted cross-edge; as $\texttt{\#}K_{r+1}(G)$ is at least the number of survivors, the
result follows.

Fix a class-edge $e\subseteq V_i$. Each canonical copy of $K_{r+1}$ through $e$ is determined by picking one vertex
from each $V_j$ with $j\ne i$, giving $\prod_{j\ne i}n_j$ of them. By \eqref{eq-DB}, we have 
$\sum_i(n_i-\tfrac nr)^2=O_r(D)$, so $\max_{i} |n_i -\tfrac nr| =O_r(\sqrt D)$. The sizes of parts of 
the Tur\'an graph $T_{n,r}$ satisfy $t_i=\tfrac nr+O_r(1)$. Therefore, we conclude that 
\[
        \prod_{j\ne i}n_j=\Bigl(\frac nr\Bigr)^{r-1}+O_r\bigl( \sqrt D\cdot n^{r-2}\bigr) 
\]
and
\[   c(n,K_{r+1})=\Bigl(\frac nr\Bigr)^{r-1}+O_r(n^{r-2}). \]
Therefore, we have $\prod_{j\ne i}n_j\ge c(n,K_{r+1})-C_r(\sqrt D+1)n^{r-2}$. Summing over the $s$
class-edges, the number of such canonical copies is at least
$s\,c(n,K_{r+1})-C_r\,s(\sqrt D+1)n^{r-2}$.

Fix a deleted cross-edge $xy$ with $x\in V_a$, $y\in V_b$ ($a\ne b$), and a class-edge $e$; a
canonical copy through $e$ containing $xy$ contains both $x$ and $y$. If $e\subseteq V_a$,
then $x$ must be an endpoint of $e$ and $y$ the chosen vertex of $V_b$, leaving the $r-2$
parts other than $V_a,V_b$ free (the case $e\subseteq V_b$ is symmetric); otherwise $e,x,y$
are all determined, leaving $r-3$ parts free. Either way there are $O_r(n^{r-2})$ such copies.
Summing over the $s$ class-edges and $h$ deleted cross-edges, at most $O_r(h\,s\,n^{r-2})$ canonical copies are lost. 
Consequently,  we obtain 
$\texttt{\#} K_{r+1}(G)\ge s\,c(n,K_{r+1})-C_r\,s(\sqrt D+1)n^{r-2}-O_r(h\,s\,n^{r-2})$, which
after enlarging the constant $C_r$ gives the desired bound of $\texttt{\#} K_{r+1}(G)$.
\end{proof}

\subsection{Proof of Theorem \ref{thm-cliques}}

We prove the lower bound of Theorem \ref{thm-cliques}. 
For any small $\varepsilon >0$, suppose that $1\le q\le(\delta_r-\varepsilon)\sqrt n$ and $G$  is an $n$-vertex graph with 
$\lambda(G)\ge\lambda(Y_{n,r,q})$ and $\texttt{\#} K_{r+1}(G)\le q\,c(n,K_{r+1})$. Our goal is to prove that $G\in \mathcal{T}_{n,r,q}$, that is, $G$ is obtained from the Tur\'{a}n graph $T_{n,r}$ by adding $q$ class-edges.   
By Lemma \ref{lem:reduction}, $G$ is obtained from 
a complete $r$-partite graph $K=K_r(n_1,\ldots,n_r)$ by adding $s$ class-edges and deleting $h$ cross-edges. Writing
$D:=e(T_{n,r})-e(K)\ge 0$, we have 
\[
        s,h,D=O_r(q)=O_r(\sqrt n).
\]
Denote $\alpha:=1- \frac{1}{r}$. 
Combining with Theorem \ref{thm:local-spectral}, 
we get  
   \[     \lambda(G)\le \lambda(T_{n,r})-\frac{2\alpha D}{n}+\frac{2(s-h)}{n}
        +\frac{(s^2+h^2)\alpha^{-1}}{n^2}+o(n^{-1}). \]  
Lemma \ref{lem:Y-expansion} gives
$\lambda(Y_{n,r,q})=\lambda(T_{n,r})+2q/n+O_r(n^{-3/2})$. Then 
\begin{equation}      
\label{eq:key-diff}
        0\le \lambda(G)-\lambda(Y_{n,r,q})
        \le \frac{2(s-q-h)}{n}-\frac{2\alpha D}{n}
            +\frac{(s^2+h^2)\alpha^{-1}}{n^2}+o(n^{-1}). 
\end{equation} 
Next, we divide the remaining arguments into three cases. 

\subsection*{Case 1. $s\le q-1$}
Since $h=O_r(\sqrt n)$, for large $n$, we get 
\[
        -\frac{2h}{n}+\frac{h^2 \alpha^{-1}}{n^2}=\frac hn 
        \Bigl(\frac{h\alpha^{-1}}{n}-2\Bigr)\le 0.
\]
Dropping the non-positive terms involving $h$ and $D$, we see that \eqref{eq:key-diff} gives
$$ \lambda(G)-\lambda(Y_{n,r,q})\le 
\frac{2(s-q)}{n} + \frac{s^2\alpha^{-1}}{n^2}+ o(n^{-1}). $$ 
Since the right-hand side above is increasing in
$s$ and $q\le (\delta_r -\varepsilon)\sqrt{n}$, we get 
\[
        \lambda(G)-\lambda(Y_{n,r,q})
        \le -\frac2n+\frac{(q-1)^2 \alpha^{-1}}{n^2}+o(n^{-1})
        \le \frac{-2 + (\delta_r- \varepsilon)^2 \alpha^{-1}}{n} + o(n^{-1})<0, 
\]
where the last inequality holds since $\delta_r= \sqrt{2}\alpha$. This contradicts $\lambda(G)\ge\lambda(Y_{n,r,q})$.

\subsection*{Case 2. $s=q$}
In this case, the inequality \eqref{eq:key-diff} reads
$$
\lambda(G)-\lambda(Y_{n,r,q})\le \frac{-2h}{n} - \frac{2\alpha D}{n}+\frac{(q^2+h^2)\alpha^{-1}}{n^2}+o(n^{-1}). $$
If $h=0$ and $D=0$, then $e(K)=e(T_{n,r})$, so $K$ is balanced, i.e., $K=T_{n,r}$. Thus, $G$ is obtained from the Tur\'{a}n graph  $T_{n,r}$ by adding $q$  class-edges, which is the desired extremal case.

Next, we assume that $h\ge1$ or $D\ge1$. In either case, we have 
\[
        -\frac{2h}{n}-\frac{2\alpha D}{n}+ \frac{h^2 \alpha^{-1}}{n^2}\le-\frac{2\alpha}{n}+o(n^{-1}).
\]
Indeed, if $D\ge1$, this follows from $-2\alpha D/n\le-2\alpha/n$ and $-2h/n+h^2 \alpha^{-1}/n^2\le0$. If $h\ge1$, then $-2h/n+h^2\alpha^{-1}/n^2$ is decreasing on $h\in [1,O_r(\sqrt n\, )]$, so its value is at most $-2/n+\alpha^{-1}/n^2$, which is at most $-2\alpha/n+o(n^{-1})$ because $\alpha = 1- \frac{1}{r} <1$. Therefore,  we get 
\[
        \lambda(G)-\lambda(Y_{n,r,q})\le-\frac{2\alpha}{n} 
        + \frac{q^2 \alpha^{-1}}{n^2}+o(n^{-1})
      <0
\]
for large $n$, again a contradiction.

\subsection*{Case 3. $s=q+\ell$ with $\ell\ge1$}
The assumption gives $\texttt{\#} K_{r+1}(G)\le q\,c(n,K_{r+1})$. 
By Lemma \ref{lem:local-counting}, we have 
\[
       \texttt{\#} K_{r+1}(G) \ge s\,c(n,K_{r+1})-C_r\,s(\sqrt D+h+1)n^{r-2}. 
\]
Combining with the two bounds,  we obtain $\ell\,c(n,K_{r+1})\le C_r\,s(\sqrt D+h+1)n^{r-2}$, which together with 
$c(n,K_{r+1})=(n/r)^{r-1}+O_r(n^{r-2})$ yields 
$ \sqrt D+h+1\ge \kappa \frac{\ell n}{s}$ for some $\kappa=\kappa(r)>0$.
Since $D=e(T_{n,r})-e(K)$ is a non-negative integer, we have $D \ge \sqrt{D} $. Then  
\[
        \frac{2h+2\alpha D}{n}\ge 
       \frac{ 2\alpha h+2\alpha\sqrt D}{n}  \ge  \frac{2\alpha \kappa \ell}{s} - \frac{2\alpha}{n}.   
\]
Substituting into \eqref{eq:key-diff}, where $s-q=\ell$ and
$\frac{(s^2+h^2)\alpha^{-1}}{n^2} =O_r(n^{-1})$, we get 
\[
        \lambda(G)-\lambda(Y_{n,r,q})
        \le \frac{2\ell}{n}- \frac{2\alpha \kappa \ell}{s}+O_r(n^{-1}).
\]
As $s=O_r(\sqrt n\,)$ and $\ell \ge 1$, we have $\frac{2\ell}{n} < \frac{\alpha\kappa \ell}{s}$ for large $n$, so
$\frac{2\ell}{n} - \frac{2\alpha\kappa\ell}{s} < - 
\frac{\alpha \kappa\ell}{s}$. This term is of order $n^{-1/2}$ and dominates the $O_r(n^{-1})$ remainder, so
$\lambda(G)<\lambda(Y_{n,r,q})$ for large $n$, a contradiction.

\medskip
In conclusion, we obtain that for every integer $1\le  q\le(\delta_r-\varepsilon)\sqrt n$, if $G$ is a graph with 
$\lambda(G)\ge\lambda(Y_{n,r,q})$ and 
 $\texttt{\#} K_{r+1}(G)\le q\,c(n,K_{r+1})$, then $s=q$, and $G$ is obtained from $T_{n,r}$ by adding $q$ edges. 
Each class-edge added to $T_{n,r}$ lies in at least $c(n,K_{r+1})$ copies of $K_{r+1}$, and distinct class-edges give distinct copies. 
Hence, we conclude that $G\in \mathcal{T}_{n,r,q}$ and $\texttt{\#} K_{r+1}(G) =  q\,c(n,K_{r+1})$.

\section{The sharp threshold for odd cycles}

\label{sec:odd}

First of all, we show that the threshold $\delta_{\rm opt} =1/\sqrt{2}$ in Theorem \ref{thm-odd-cycles} is best possible. 

\begin{example}\label{exam-odd-cycle}
Fix $\varepsilon>0$ and let $N$ be sufficiently
large. Let $q$ be an integer with
$(1+\varepsilon)\sqrt N\le q\le C \sqrt N$ for some constant $C$, and let
$H_{N,q}$ be obtained from $K_{N+1,N-1}$ by adding a
star with $q$ edges inside the part of size $N+1$.
Then $\lambda(H_{N,q})>\lambda(Y_{2N,2,q})$ and
$\texttt{\#}\,C_{2k+1}(H_{N,q})<q\,c(2N,C_{2k+1})$.
\end{example} 

\begin{proof}
By Proposition \ref{prop:construction}, we have 
\[
        \lambda(H_{N,q})
        \ge
        N+\frac qN+\frac{q^2-N}{2N^2}+O(N^{-3/2}).
\]
Applying Lemma \ref{lem:Y-expansion}, it follows that 
\[
        \lambda(Y_{2N,2,q})=N+\frac qN+O(N^{-3/2}).
\]
Therefore, we get 
\[
        \lambda(H_{N,q})-\lambda(Y_{2N,2,q})
        \ge
        \frac{q^2-N}{2N^2}+O(N^{-3/2}).
\]
If $q\ge(1+\eps)\sqrt N$, then
$q^2-N\ge \bigl( 2\varepsilon + \varepsilon^2 \bigr)N$. 
So the main term above is a positive constant times $N^{-1}$, while the error term is $O(N^{-3/2})$.  Thus, $ \lambda(H_{N,q})>\lambda(Y_{2N,2,q})$ for all sufficiently large $N$. 

Write $L$ and $R$ for the two parts of $K_{N+1,N-1}$, with $|L|=N+1$ and $|R|=N-1$. The $q$ star edges lie in $L$ and share a common center $v$. Call an edge of $H_{N,q}$ a \emph{cross edge} if it joins $L$ and $R$, and a \emph{star edge} otherwise.
First, we see that 
each copy of $C_{2k+1}$ uses exactly one star edge. 
Walk once around such a cycle. A cross edge switches the side between $L$ and $R$, 
while a star edge stays on the same side. Since the walk returns to its start, the number of switches is even, so the cycle uses an even number of cross edges. As the length equals the number of cross edges plus  the number of star edges, the cycle uses an \emph{odd} number of star edges. On the other hand, every vertex of a cycle has degree $2$, so the center $v$ lies on at most two of the used star edges; hence the cycle uses at most two. The odd number that is at most two equals one.

Fix a star edge $xy$, with $x,y\in L$. A copy of $C_{2k+1}$ through $xy$ consists of $xy$ together with a path of length $2k$ from $x$ to $y$ alternating between $L$ and $R$: $ x,b_1,a_1,b_2, \ldots,a_{k-1},b_k,y$, 
where $b_1,\ldots,b_k\in R$ are distinct and $a_1,\ldots,a_{k-1}\in L\setminus\{x,y\}$ are distinct. Each such copy corresponds to exactly one such sequence (read from $x$ to $y$). A sequence is fixed by an ordered choice of the $b_i$ from $R$ ($N-1$ vertices) and an ordered choice of the $a_i$ from $L\setminus\{x,y\}$ (also $N-1$ vertices), giving $(N-1)_k\,(N-1)_{k-1}$ copies through $xy$. Summing over the $q$ star edges, we get 
\[
    \texttt{\#}\, C_{2k+1}(H_{N,q})=q\,(N-1)_k\,(N-1)_{k-1}.
\]
Similarly, by adding an edge inside a part of $K_{N,N}$, we have 
$$c(2N,C_{2k+1})=(N)_k\,(N-2)_{k-1}. $$ 
Hence 
\[
    \frac{(N-1)_k\,(N-1)_{k-1}}{(N)_k\,(N-2)_{k-1}}
    =\frac{N-k}{N}\cdot\frac{N-1}{N-k}
    =\frac{N-1}{N}<1,
\]
so $(N-1)_k(N-1)_{k-1}<c(2N,C_{2k+1})$. Multiplying by $q$ gives $\texttt{\#}\,{C_{2k+1}}(H_{N,q})<q\,c(2N,C_{2k+1})$.
\end{proof}

\subsection{Stability and counting for odd cycles}

Applying Theorem \ref{thm:FLLM}, 
we get the following stability reduction.

\begin{lemma} \label{lem-stability-odd}
Fix $\eps>0$ and suppose that 
$ n=2N+\sigma$, where $\sigma\in\{0,1\}$,  
and $q$ is an integer such that $   1\le q\le (1-\eps)\sqrt N$. If $G$ is an $n$-vertex graph with 
\[
        \lambda(G)\ge \lambda(Y_{2N+\sigma,2,q}),
\]
and 
\[
       \texttt{\#}\,{C_{2k+1}}(G)\le q\,c(n,C_{2k+1}),
\]
then 
there is a partition $ V(G)=A\sqcup B$
such that $G$ is obtained from the complete bipartite graph $K_{|A|,|B|}$ by adding $s$ edges inside the two parts and deleting $h$ cross-edges, where 
$ s,h=O_k(q)$. Moreover, if we write $ |A|=N+\sigma+a$ and $ |B|=N-a$, where $a \ge 0$, then $a=O_k(\sqrt{s})$. 
\end{lemma}

\begin{proof}
Let $\delta_F,\eta_F$ be the constants of Theorem
\ref{thm:FLLM} for $F=C_{2k+1}$ and, say,
$\varepsilon_1=\tfrac14$. 
Clearly, we have $\lambda(G)\ge \lambda (Y_{n,2,q}) \ge  \lambda(T_{n,2})$. Moreover,
$q\le(1-\eps)\sqrt N\le\delta_F n$ and
$\texttt{\#}\,C_{2k+1}(G)\le q\,c(n,C_{2k+1})
\le(q+\varepsilon_1)\,c(n,C_{2k+1})$.
Applying Theorem \ref{thm:FLLM} with $r=2$, we get a
partition $V(G)=A\sqcup B$ such that $G$ is obtained
from $K_{|A|,|B|}$ by adding $s$ class-edges and
deleting $h$ cross-edges, where
$h\le 2s\le 4q/\eta_F$, so $s,h=O_k(q)$, and
$\bigl||A|-|B|\bigr|\le 4\sqrt s$. We may assume
$|A|\ge|B|$, so $|A|\ge\lceil n/2\rceil=N+\sigma$.
Writing $|A|=N+\sigma+a$ and $|B|=N-a$, we get
$a\ge0$ and $2a+\sigma=|A|-|B|\le 4\sqrt s$, hence
$a\le 2\sqrt s$.
\end{proof}

The following result is a direct consequence of Theorem \ref{thm:local-spectral} in the case $r=2$. 

\begin{lemma}\label{lem:local-spectrum-odd}
Let $n=2N+\sigma$, where $\sigma\in\{0,1\}$.  Let $G$ be obtained from $ K_{p,b}$, where 
$p=N+\sigma+a$ and $ b=N-a$, 
by adding $s$ edges inside the two parts and deleting $h$ cross-edges.  Assume
\[
        a=O_k(N^{1/4}),
        \qquad
        s=O_k(\sqrt N),
        \qquad
        h=O_k(\sqrt N).
\]
Then
\[
        \lambda(G)
        \le
        \sqrt{N(N+\sigma)}
        +\frac{s-h}{N}
        -\frac{a(a+\sigma)}{2N}
        +\frac{s^2+h^2}{2N^2}
        +o(N^{-1}).
\]
\end{lemma}

\begin{proof}
Apply Theorem \ref{thm:local-spectral} with $r=2$ and
$\alpha=\tfrac12$ to $K=K_{p,b}$, where
$n=p+b=2N+\sigma$. Note that
$\lambda(T_{n,2})=\sqrt{N(N+\sigma)}$ exactly, since
$T_{n,2}=K_{N+\sigma,N}$. Also
\[
  D=e(T_{n,2})-e(K_{p,b})
   =N(N+\sigma)-(N+\sigma+a)(N-a)
   =a(a+\sigma),
\]
so $D=O_k(\sqrt N)$ by $a=O_k(N^{1/4})$, and the
hypotheses of Theorem \ref{thm:local-spectral} hold.
The theorem yields
\[
  \lambda(G)\le\sqrt{N(N+\sigma)}
   -\frac{a(a+\sigma)}{2N+\sigma}
   +\frac{2(s-h)}{2N+\sigma}
   +\frac{2(s^2+h^2)}{(2N+\sigma)^2}
   +O_k(N^{-5/4}).
\]
Since $a(a+\sigma),\,s+h=O_k(\sqrt N)$ and
$\sigma\in\{0,1\}$, replacing $2N+\sigma$ by $2N$ in
each term changes it by at most $O_k(N^{-3/2})$,
which is absorbed into $o(N^{-1})$. This gives the
stated bound.
\end{proof}

The next lemma is the odd-cycle analogue of Lemma \ref{lem:local-counting}. 

\begin{lemma}\label{lem:local-counting-odd}
Under the condition of Lemma \ref{lem:local-spectrum-odd}, there exists a constant $C_k'>0$ such that
\[
       \texttt{\#}\, {C_{2k+1}}(G)
        \ge
        s\,c(2N+\sigma,C_{2k+1})
        -C_k' s(a+h+1)N^{2k-2}.
\]
\end{lemma}

\begin{proof}
We follow the proof of Lemma \ref{lem:local-counting}, replacing the
product over the other parts by the alternating-path
count of Example \ref{exam-odd-cycle}. 
Call a copy of $C_{2k+1}$ \emph{canonical} if it uses
exactly one class-edge. By the parity argument of Example
\ref{exam-odd-cycle}, every copy uses at least one
class-edge, so it suffices to lower bound the surviving
canonical copies, and each is counted once by its unique
class-edge.

Fix a class-edge $e=xy$ in a part $W$ of $K_{p,b}$, with
opposite part $Z$. The canonical copies through $e$ are
$e$ plus an alternating $x$--$y$ path of length $2k$, of
which there are $(|Z|)_k(|W|-2)_{k-1}$ in $K_{p,b}$. Since
$a=O_k(N^{1/4})$, every part has size $N+O_k(N^{1/4})$, so
this equals $c(2N+\sigma,C_{2k+1})+O_k((a+1)N^{2k-2})$.

A deleted cross-edge $uv$ kills only canonical copies that
use it as a path-edge; fixing its position and an endpoint
leaves $\le 2k-2$ free vertices, so at most
$O_k(N^{2k-2})$ copies through each $e$. Summing over the
$s$ class-edges and $h$ deleted cross-edges and enlarging
the constant gives
\[
    \texttt{\#}\,C_{2k+1}(G)
    \ge s\,c(2N+\sigma,C_{2k+1})
        -C_k'\,s(a+h+1)N^{2k-2}. \qedhere
\] 
\end{proof}

\subsection{Proof of Theorem \ref{thm-odd-cycles}}

\label{sec:odd-cycle-lower}

Since $\sqrt{n} =\sqrt{2N+\sigma} = \sqrt{2N}(1+o(1))$, the condition $q\le (1/ \sqrt{2} - \eps)\sqrt{n}$ implies 
$q\le (1- \eps' )\sqrt{N}$ for some $\eps' >0$, so it suffices to show that for any fixed $k\ge 1$ and $\eps>0$, and for all sufficiently large $N$, if $ q\le (1-\eps)\sqrt N$  
and $G$ is a graph of order $n=2N+\sigma$ (where $\sigma \in \{0,1\}$) with $\lambda(G)\ge \lambda(Y_{n,2,q})$ and $ \texttt{\#}\, {C_{2k+1}}(G)\le q\, c(n, C_{2k+1})$, then  $G$ is obtained from $T_{n,2}$ by adding $q$ edges. First,  
 Lemma \ref{lem-stability-odd} implies that 
$G$ admits a partition $V(G)=A\sqcup B$ satisfying 
$|A|=N+\sigma +a$ and $|B|=N- a$, where $a=O_k(N^{1/4})$. Moreover, writing  
$s=e(A) + e(B)$ and $h=|A||B| - e(A,B)$, we have  
$ s,h=O(\sqrt N)$. 
 Combining Lemma \ref{lem:local-spectrum-odd} and  Lemma \ref{lem:Y-expansion} gives
\begin{equation} \label{eq-diff-G-Y}
        \lambda(G)-\lambda(Y_{n,2,q})
        \le
        \frac{s-q-h}{N}
        -\frac{a(a+\sigma)}{2N}
        +\frac{s^2+h^2}{2N^2}
        +o(N^{-1}).
\end{equation}
We now split the remaining argument into three cases.

\subsection*{Case 1. $s\le q-1$}

Since $h=O(\sqrt N)$, for large $N$, we have
$ -\frac hN+\frac{h^2}{2N^2}\le 0$. 
Also $a(a+\sigma)\ge0$.  Thus the right side of (\ref{eq-diff-G-Y}) is maximized by taking $s=q-1$, $h=0$, and $a=0$.  Therefore
\[
\begin{aligned}
        \lambda(G)-\lambda(Y_{n,2,q})
        \le
        -\frac1N+\frac{(q-1)^2}{2N^2}+o(N^{-1}) 
        \le
        -\frac1N+\frac{q^2}{2N^2}+o(N^{-1}).
\end{aligned}
\]
By the assumption $q\le (1- \varepsilon) \sqrt{N}$, we have $  \frac{q^2}{2N^2}\le \frac{(1-\eps)^2}{2N}$. 
Hence
\[
        \lambda(G)-\lambda(Y_{n,2,q})
        \le
        \left(-1+\frac{(1-\eps)^2}{2}\right)\frac1N+o(N^{-1})<0,
\]
which contradicts the assumption $\lambda(G)\ge\lambda(Y_{n,2,q})$. 

\subsection*{Case 2. $s=q$}

If $a=0$ and $h=0$, then $G$ is obtained from $K_{N+\sigma, N}$ by adding $s=q$ edges. 
This is the desired extremal case.  
Next, we consider the case $ a+h\ge 1$. 
Since $s=q$, we get from (\ref{eq-diff-G-Y}) that  
\[  \lambda(G)-\lambda(Y_{n,2,q})
        \le
        -\frac hN-\frac{a(a+\sigma)}{2N}
        +\frac{q^2+h^2}{2N^2}+o(N^{-1}).
\]
Recall that $h=O(\sqrt N)$. We have $  -\frac hN+\frac{h^2}{2N^2}\le -\frac{h}{2N}$ 
for large $N$.  Since $a$ and $h$ are non-negative integers and $a+h\ge1$, we get 
$  \frac{h}{2}+\frac{a(a+\sigma)}2\ge \frac12$. 
Therefore, it follows that 
\[
        \lambda(G)-\lambda(Y_{n,2,q})
        \le
        -\frac1{2N}+\frac{q^2}{2N^2}+o(N^{-1}).
\]
Since $q\le (1- \eps) \sqrt{N}$, we get $  -\frac1{2N}+\frac{q^2}{2N^2}
        \le
        -\frac{1-(1-\eps)^2}{2N}<0$. 
This is again a contradiction.

\subsection*{Case 3. $s\ge q+1$}

Write $s=q+\ell$ with $\ell\ge1$. The assumption gives $\texttt{\#}\, {C_{2k+1}}(G)\le qc(n,C_{2k+1})$, while Lemma \ref{lem:local-counting-odd} gives $\texttt{\#}\,{C_{2k+1}}(G)\ge s c(n,C_{2k+1}) - C_k' s(a+h+1)N^{2k-2}$. Combining the two, we obtain 
\[
    qc(n,C_{2k+1})\ge (q+\ell)c(n,C_{2k+1})-C_k' s(a+h+1)N^{2k-2},
\]
so $\ell c(n,C_{2k+1})\le C_k' s(a+h+1)N^{2k-2}$. Note that $c (n,C_{2k+1}) =\Omega_k(N^{2k-1})$ for sufficiently large $N$. Substituting this and dividing by $C_k' sN^{2k-2}$ yields a constant $c_k>0$ satisfying 
\[
    a+h\ge  \frac{c_k \ell N}{s}-1.
\]

We now bound the negative terms in the spectral estimate \eqref{eq-diff-G-Y}. 
Since $a(a+\sigma)\ge a$ when $a\ge 0$, we have 
$ h+\frac{a(a+\sigma)}2\ge\frac{a+h}2\ge\frac12\big(c_k\frac{\ell N}{s}-1\big)$. 
Dividing by $N$ gives 
$$\dfrac hN+\dfrac{a(a+\sigma)}{2N}\ge\dfrac{c_k}{2}\dfrac{\ell}{s}-\dfrac1{2N}. $$
Applying \eqref{eq-diff-G-Y} with $s-q=\ell$, we have 
\[
    \lambda(G)-\lambda(Y_{n,2,q})
    \le\frac{\ell}{N}-\frac{c_k}{2}\frac{\ell}{s}
    +\frac1{2N}+\frac{s^2+h^2}{2N^2}+o(N^{-1}).
\]
By the structural reduction $s,h=O_k(\sqrt N)$, so the last three terms are $O_k(N^{-1})$ and $s\le C_k''\sqrt N$ for some constant $C_k''>0$. For large $N$, we have $\frac{c_k}{2s}-\frac1N\ge\frac{c_k}{4s}$. Hence
\[
    \frac{\ell}{N}-\frac{c_k}{2}\frac{\ell}{s}
    =-\ell\Big(\frac{c_k}{2s}-\frac1N\Big)
    \le-\frac{\ell c_k}{4s} 
    \le-\frac{c_k}{4C_k''\sqrt{N}}.
\]
Therefore, we obtain 
\[
    \lambda(G)-\lambda(Y_{n,2,q})\le-\Omega_k(N^{-1/2})+O_k(N^{-1})<0
\]
for large $N$, contradicting $\lambda(G)\ge\lambda(Y_{n,2,q})$. Thus, Case 3 is impossible.

In conclusion, if $\texttt{\#}C_{2k+1}(G)\le
q\,c(n,C_{2k+1})$, then $s=q$ and $a=h=0$, so $G$ is
obtained from $T_{n,2}$ by adding $q$ class-edges.
Each class-edge lies in at least $c(n,C_{2k+1})$
 copies of $C_{2k+1}$, and distinct
class-edges give distinct copies. Hence, we get $G\in\mathcal T_{n,2,q}$ and 
$\texttt{\#}C_{2k+1}(G)\ge q\,c(n,C_{2k+1})$.

\section{Concluding remarks}

\label{sec-concluding-remark}

In this paper, we studied the 
spectral supersaturation problems for cliques and
odd cycles, and determined the exact thresholds $\delta_{\rm spec}(K_{r+1})$ and
$\delta_{\rm spec}(C_{2k+1})$. 
We proved that the clique threshold is
$(1-\frac1r)\sqrt2$ for every $r\ge 2$ (Theorem \ref{thm-cliques}), and the odd-cycle threshold is the
universal value $1/\sqrt2$ for every $k\ge 1$ (Theorem
\ref{thm-odd-cycles}). These are the spectral supersaturation results with both the exact count $q\,c(n,F)$ and the optimal range of $q$. 
A common theme of
these spectral results is the sharp contrast with the edge setting of Lov\'asz--Simonovits and Pikhurko--Yilma in Theorems \ref{thm-edge-clique} and \ref{thm-edge-cycle}, respectively. For example,
 the range of $q$ drops down from order $n$ to order $\sqrt n$, and the clique threshold changes from
$1/r$ to $(1 \!-\! \frac1r)\sqrt2$. 
The proofs rely on the spectral supersaturation-stability method, together with a second-order spectral estimate (Theorem \ref{thm:local-spectral})  for nearly Tur\'an graphs, which controls the spectral radius of a complete $r$-partite graph after $O_r(\sqrt n)$ edge changes and may be of independent interest. As a further application, we proved a spectral supersaturation with a continuous additive gap (Theorem \ref{thm-continuous}) that bounds $\texttt{\#}F(G)$ by
the spectral increment $\lambda(G)-\lambda(T_{n,r})$. The clique case improves the
Bollob\'as--Nikiforov bound just above the Tur\'an threshold.

\smallskip 
Building on these results, we conclude with some open problems for further study.

\medskip
\noindent\textbf{Determining thresholds for all
color-critical graphs.} 
Recall that 
$\delta_{\mathrm{spec}}(F)$  is defined to be the largest constant $\delta_F$
such that for every $\eps>0$, all sufficiently large
$n$ and every integer $1\le q\le
(\delta_F -\eps)\sqrt n$, each $n$-vertex
graph $G$ with $\lambda(G)\ge\lambda(Y_{n,r,q})$ contains at
least $q\,c(n,F)$ copies of $F$. 
We conjecture that the spectral threshold $\delta_{\mathrm{spec}}(F)$ depends on $F$ only through its chromatic number, and 
this threshold may be governed by whether a single added structure can raise the spectral radius above $\lambda(Y_{n,r,q})$. 
This would contrast with the edge setting, where
the threshold $\delta_{\mathrm{edge}}(F)$ does depend on the  structural properties of $F$: it equals
$1/r$ for $K_{r+1}$, but $(r-1)/r^2$ for $K_{r+2}-e$, the clique with one edge removed \cite{PY2017}. 
A spectral threshold independent of the structure of $F$ 
would be a different phenomenon. 
The graph $K_{r+2}-e$ is
the natural first test case. 
We speculate that $K_{r+2}-e$ may have the same spectral threshold as $K_{r+1}$ despite a different edge threshold. 

\medskip
\noindent\textbf{The linear range of the increment.}
Our results for cliques and odd cycles are located in the range $q< \delta_{F} \sqrt n$. 
It would be interesting to describe the behavior above the threshold, say $q \gg  \sqrt n$, where the candidate graph $Y_{n,r,q}$ is no longer extremal, 
and the intermediate
range $\sqrt n\ll q\ll n$ is widely open. Since the
condition $\lambda(G)\ge\lambda(Y_{n,r,q})$ no longer forces $\texttt{\#}F(G)\ge q\,c(n,F)$ in this range, it would be worth determining the new 
candidate graph and the correct order of $\texttt{\#}F(G)$.

Let $T_{n,r,q}$ be obtained from Tur\'{a}n's graph $T_{n,r}$ by adding a star with $q$ edges into a largest part. 
Fang, Li, Lin and Ma \cite[Theorem 7.2]{FLLM2025} showed that the range of $q$ can be extended to be linear when taking $T_{n,r,q}$ as a candidate graph, that is,  
there is a constant $\delta_F^*>0$ such that for every integer $1\le q\le\delta_F^*n$,
the condition $\lambda(G)\ge\lambda(T_{n,r,q})$ forces
$\texttt{\#}F(G)\ge q\,c(n,F)$. It is interesting to determine the new threshold $\delta_F^*$ for every color-critical graph $F$.  
We conjecture that the threshold $\delta_F^*$ for cliques and odd cycles would be consistent with those in the edge setting; see Table \ref{tab:edge-vs-spec}.  

For Theorem \ref{thm-continuous}, we showed a continuous version for counting color-critical graphs in any graph $G$ with $\lambda (G)\ge \lambda (T_{n,r}) \, +\, C$, where $0< C \le \delta_F {n}^{-1/2}$. Naturally, it is interesting to count substructures in the host graph with arbitrary  spectral density above $\lambda (T_{n,r})$. 
For example, one could study the spectral supersaturation in the remaining  range $C\gg {n}^{-1/2}$. 
The problem of counting cliques under the edge density condition was already solved by Reiher \cite{Rei2016}.

\medskip
\noindent\textbf{Counting degenerate bipartite graphs.}
Our framework is tailored to color-critical graphs, for
which the Tur\'an number is $\Theta(n^2)$ and $T_{n,r}$ is
the unique extremal graph; this is what makes $Y_{n,r,q}$
the right comparison graph. For degenerate $F$, such as $C_{2t}$ or $K_{t,t}$, the Tur\'an number is $o(n^2)$ and the extremal
graphs are sparse, so neither $T_{n,r}$ nor $Y_{n,r,q}$
has an obvious analogue, and even the correct order of the
threshold $q$ is unclear. Building a spectral
supersaturation theory for degenerate $F$, relative to the
appropriate sparse extremal structures, appears to require new ideas.

Here, we present a possible problem for counting $C_{2t}$. Let $$S_{t-1,n}^+:=K_{t-1}\vee I_{n-t+1}^+$$
be the $n$-vertex graph obtained from the join $K_{t-1}\vee I_{n-t+1}$ by adding an edge into its independent set. A breakthrough of 
Cioab\u{a}, Desai and Tait \cite{CDT2024} shows  that for any fixed $t\ge 3$, if $n$ is sufficiently large and  $G$ is an $n$-vertex graph with $\lambda (G) > \lambda (S_{t-1,n}^+)$, then $G$ contains a $2t$-cycle. 
We refer the readers to \cite{BDT2025} for a general spectral theorem for degenerate bipartite graphs. 
Inspired by the above result, we propose the following spectral supersaturation problem for $C_{2t}$ for every $t\ge 3$. 

\begin{problem}[Counting even cycles]
For every integer $t\ge 3$, does there exist a constant $c(t) >0$ such that if $n$ is sufficiently large and 
$G$ is an $n$-vertex graph with $\lambda (G) > \lambda (S_{t-1,n}^+)$, then 
$$ \texttt{\#}\, C_{2t} (G) \ \ge \ c(t)\cdot n^{t-3}\ ? $$
 If the above is true, 
 what is the largest value of such $c(t)$?  
\end{problem}

\smallskip 
Finally, a complementary direction replaces the vertex-spectral 
condition $\lambda(G)\ge\lambda(T_{n,r})$ by the
edge-spectral condition $\lambda^2(G)\ge {(1-\frac{1}{r})2m}$, in the spirit
of Nikiforov's theorem \cite{Nikiforov}. 
For counting triangles and books, this scale
is the natural one and has been well-studied; we refer the readers to \cite{NZ2021, ZLL2026, ZZ2026, ZYZZ2026}.  
It would be interesting to develop the corresponding supersaturation counts and 
to connect both the two scales for general graphs $F$; see  
 \cite{LLZ-edge-spectral, LLZ-edge-color-critical, LLZ2026-Spec-Sido, LLS2026-odd-cycles} for recent developments.

\section*{Acknowledgment} 
The corresponding author would like to thank Prof.~Bo Ning for raising the optimal spectral supersaturation problem after a talk at Anhui University in May 2025. 
Thanks also go to Dr. Longfei Fang for helpful discussions during the study of the spectral graph problem over the past few years.
In an early exploratory stage of the project, 
the authors additionally used language-model-based tools to brainstorm candidate proof strategies. 
The authors developed the arguments, wrote and verified all proofs.  
Any suggestions arising from these tools served only as informal inspiration.

\end{document}